\documentclass[11pt]{amsart}

\usepackage{amsmath,amssymb,amsfonts,amsthm,bbm}
\usepackage{hyperref}
\usepackage{graphicx}
\usepackage{chngcntr}
\counterwithin{figure}{section}
\date{}
\usepackage{enumerate}
\usepackage{fullpage}

\newcommand{\norm}[1]{\lVert#1\rVert}

\newcommand{\re}{\mathop{\mathrm{Re}}}
\newcommand{\im}{\mathop{\mathrm{Im}}}
\newcommand{\tr}{\mathop{\mathrm{Tr}}}

\newtheorem{theorem}{Theorem}[section]
\newtheorem{corollary}[theorem]{Corollary}
\newtheorem{proposition}[theorem]{Proposition}

\theoremstyle{definition}
\newtheorem{definition}[theorem]{Definition}
\newtheorem{remark}[theorem]{Remark}
\newtheorem{lemma}[theorem]{Lemma}

\numberwithin{equation}{section}\title{Mesoscopic linear statistics of Wigner matrices}

\author{A. Lodhia}
\address{Department of Mathematics, Massachusetts Institute of Technology, Cambridge, MA, USA}
\email{lodhia@mit.edu}

\author{N. J. Simm}
\address{School of Mathematical Sciences, Queen Mary University of London, London E1 4NS, UK}
\email{n.simm@qmul.ac.uk}

\begin{document}
\begin{abstract}
We study linear spectral statistics of $N \times N$ Wigner random matrices $\mathcal{H}$ on \textit{mesoscopic scales}. Under mild assumptions on the matrix entries of $\mathcal{H}$, we prove that after centering and normalizing, the trace of the resolvent $\tr(\mathcal{H}-z)^{-1}$ converges to a stationary Gaussian process as $N \to \infty$ on scales $N^{-1/3} \ll \im(z) \ll 1$ and explicitly compute the covariance structure. The limit process is related to certain regularizations of fractional Brownian motion and logarithmically correlated fields appearing in \cite{FKS13}. Finally, we extend our results to general mesoscopic linear statistics and prove that the limiting covariance is given by the $H^{1/2}$-norm of the test functions.
\end{abstract}
\maketitle

\section{Introduction}
The goal of this paper is to study the limiting fluctuations as $N \to \infty$ of the linear spectral statistic
\begin{equation}
 X^{\mathrm{meso}}_{N}(f) := \sum_{j=1}^{N}f(d_{N}(E-\lambda_{j})) \label{lsmeso}
\end{equation}
where $\lambda_{1},\ldots,\lambda_{N}$ are the eigenvalues of an $N \times N$ Wigner random matrix $\mathcal{H}$. The \textit{mesoscopic} or \textit{intermediate} scale is defined by the assumption that $d_{N} \to \infty$ as $N \to \infty$, but $d_{N}/N \to 0$ as $N \to \infty$. Therefore, if $f$ is decaying suitably at $\infty$, only a fraction $N/d_{N}$ of the total number of eigenvalues will contribute in the sum \eqref{lsmeso}. In recent years, there has been growing interest in understanding the limiting distribution of \eqref{lsmeso} on such mesoscopic scales. This interest has stemmed from, \textit{e.g.}, the appearance of novel stochastic processes in probability theory \cite{FKS13}, conductance fluctuations in disordered systems \cite{EK13a,EK13b} and linear statistics of the zeros of Riemann's zeta function \cite{BK14}, among others \cite{BD14,BEYY14,DK13}.

Previously, the majority of studies concentrated exclusively on the \textit{macroscopic} scale where $d_{N}=1$ and $E=0$ in \eqref{lsmeso}, denoted $X^{\mathrm{macro}}_{N}(f)$. In this case it was proved for many different types of random matrix ensembles that, provided $f$ has at least one derivative, the centered random variable
\begin{equation}
\tilde{X}^{\mathrm{macro}}_{N}(f) :=  X^{\mathrm{macro}}_{N}(f) - \mathbb{E}X^{\mathrm{macro}}_{N}(f) \label{fluct}
\end{equation}
converges in distribution to the normal law $\mathcal{N}(0,\sigma^{2})$ as $N \to \infty$. Furthermore, an explicit formula for the limiting variance $\sigma^{2}$ was obtained, see \cite{Joh98, BY05}. In analogy with classical probability, we refer to such results as \textit{central limit theorems} (CLTs).

Going to finer scales, the \textit{mesoscopic} fluctuations of \eqref{lsmeso} are known to be highly sensitive when compared to the macroscopic scale; in fact the CLT must break down if $d_{N}$ grows too quickly \cite{P06}. In particular if $d_{N}=N$, only a finite number of terms contribute in the sum \eqref{lsmeso} and we cannot expect a Gaussian limit. The latter case $d_{N}=N$ is known as microscopic and will not be considered in this article, though see \cite{E11} for an extensive review. Before we state our main results, let us describe the class of random matrices under consideration. 

\begin{definition}
\label{def:wig}
A \textit{Wigner matrix} is an $N \times N$ Hermitian random matrix $W$ whose entries $W_{ij}=\overline{W_{ji}}$ are centered, independent identically distributed complex random variables satisfying $\mathbb{E}|W_{ij}|^{2} = 1$ and $\mathbb{E}W_{ij}^{2}=0$ for all $i$ and $j$. We assume that the common distribution $\mu$ of $W_{ij}$ satisfies the sub-Gaussian decay $\int_{\mathbb{C}}e^{c|z|^{2}}d\mu(z) < \infty$ for some $c>0$. This implies that the higher moments are finite, in fact we have $\mathbb{E}|W_{ij}|^{q} < (Cq)^q$ for some $C > 0$. We denote by $\mathcal{H} = N^{-1/2}W$ the normalized Wigner matrix. Finally, in case all the entries $\mathcal{H}_{ij}$ are Gaussian distributed, the ensemble is known as the \textit{Gaussian Unitary Ensemble (GUE)}.
\end{definition}

This ensemble was introduced by Wigner who proved that in the limit $N \to \infty$, the mean eigenvalue distribution of the normalized Wigner matrix $\mathcal{H}$ converges to the semi-circle law. A modern version of this result (\textit{e.g.} Theorem 2.9 in \cite{AGZ09}) states that this convergence holds weakly almost surely, i.e.
\begin{equation}
\label{wig}
\frac{1}{N}\sum_{j=1}^{N}f(\lambda_{j}) \to \frac{1}{2\pi}\int_{-2}^{2}f(x)\sqrt{4-x^{2}}\,dx, \quad N \to \infty, \quad \mathrm{a.s.}
\end{equation}
for all bounded and continuous functions $f$. In order to state our main Theorem, we need a condition on the regularity and decay of the test functions $f$ entering in \eqref{lsmeso}. For $\alpha, \beta>0$, let $C^{1,\alpha,\beta}(\mathbb{R})$ denote the space of all functions with $\alpha$-H\"older continuous first derivative such that $f(x)$ and $f'(x)$ decay faster than $O(|x|^{-1-\beta})$ as $|x| \to \infty$. Finally, recall the notation $\tilde{X}_{N}^{\mathrm{meso}}(f) := X_{N}^{\mathrm{meso}}(f) - \mathbb{E}X_{N}^{\mathrm{meso}}(f)$.
\begin{theorem}
\label{th:maintheorem}
Let $\mathcal{H}$ be a normalized Wigner matrix as in Definition \ref{def:wig}. Suppose that $d_{N}=N^{\gamma}$ where $\gamma$ satisfies the condition $0 < \gamma < 1/3$ and consider test functions $f_{1},\ldots,f_{M} \in C^{1,\alpha,\beta}(\mathbb{R})$ for some $\alpha>0$ and $\beta>0$. Then for a fixed $E \in (-2,2)$ in \eqref{lsmeso} we have the convergence in distribution 
\begin{equation}
(\tilde{X}^{\mathrm{meso}}_{N}(f_{1}),\ldots, \tilde{X}^{\mathrm{meso}}_{N}(f_{M})) \overset{d}{\Longrightarrow} (X(f_1),\ldots,X(f_M)) \label{mesoconv}
\end{equation}
where $(X(f_1),\ldots,X(f_M))$ is an $M$-dimensional Gaussian vector with covariance matrix
\begin{equation}
\label{covar}
 \mathbb{E}(X(f_p)\overline{X(f_q)}) = \frac{1}{2\pi}\int_{-\infty}^{\infty}dk\,|k|\,\hat{f_{p}}(k)\overline{\hat{f_{q}}(k)}, \qquad 1 \leq p,q \leq M
\end{equation}
and $\hat{f}(k) := (2\pi)^{-1/2}\int_{\mathbb{R}}f(x)\,e^{-ikx}\,dx$.
\end{theorem}

This result improves and extends earlier work of Boutet de Monvel and Khorunzhy \cite{BdMK99b} who proved Theorem \ref{th:maintheorem} when $0 < \gamma < 1/8$, $M=1$ and $f(x) = (x-z)^{-1}$ (see also Theorem \ref{th:vn} below). Erd\"os and Knowles proved an analogue of Theorem \ref{th:maintheorem} for random band matrices \cite{EK13b}, including a bound on the variance of \eqref{lsmeso} in the Wigner case with the same condition $0 < \gamma < 1/3$ \cite{EK13a}. Apart from these works, CLTs for \eqref{lsmeso} were also obtained in several other ensembles \cite{BdMK99a,S00,FKS13,DK13,BD14,BEYY14}. Although these works extend to scales $0 < \gamma < 1$, the proofs rely on exact formulas for the distribution of the eigenvalues, which are unavailable in the Wigner setting. 

Let us now make some general remarks about Theorem \ref{th:maintheorem}. On macroscopic scales $\gamma=E=0$ the results are different for Wigner matrices since the limiting covariance depends on the fourth moment of the matrix entries \cite{BS04,Sh11}. On mesoscopic scales with fixed $E \in (-2,2)$, we show that this difference vanishes, indicating a particularly strong form of universality for formula \eqref{covar} (see also \cite{KKP96}). As with the local regime, the limiting distribution of \eqref{lsmeso} is \textit{universal} in the choice $E \in (-2,2)$ around which ones samples the eigenvalues. Apparently unique to the mesoscopic regime, however, is the scale invariance of the limiting Gaussian process: formula \eqref{covar} is unchanged after rescaling the arguments of the test functions by any parameter (see also Section \ref{se:mesofrac}). Optimal conditions on the test functions given in Theorem \ref{th:maintheorem} remains a significant issue ever since the seminal work of Johansson \cite{Joh98}. The latter article suggests that in the macroscopic regime, only finiteness of the limiting variance should suffice to conclude asymptotic Gaussianity, see \cite{SW13} for recent progress in this direction. In the mesoscopic regime we believe analogously that optimal conditions for asymptotic Gaussianity of \eqref{lsmeso} should be that $\int_{\mathbb{R}}|k||\hat{f}(k)|^{2}\,dk < \infty$. It is historically interesting to remark that \eqref{covar} already appeared in a famous $1963$ paper of Dyson and Mehta \cite{DM63}.

From a probabilistic viewpoint, the semi-circle law \eqref{wig} may be interpreted as the law of large numbers for the eigenvalues of a Wigner matrix; this may be considered the first natural step for the probabilist. The second natural step is to prove the CLT for the macroscopic fluctuations around \eqref{wig}. Going now to mesoscopic scales, only the first step has been investigated in detail for Wigner matrices, with the corresponding results known as \textit{local} semi-circle laws, so called because they track the convergence to the semi-circle closer to the scale of individual eigenvalues. The local semi-circle laws turned out to be a very important tool in proving the long-standing universality conjectures for eigenvalue statistics at the microscopic scale \cite{Meh04,E11,TV11b}. Consequently, a number of refinements of Wigner's semi-circle law of increasing optimality were obtained in recent years \cite{ESY09,EYY12rig,EYY12,EKTY13}. Such results will play a crucial role in our proof of the CLT at mesosopic scales.  

In order to state the local semi-circle law, it is convenient to work with the resolvent $G(z) = (\mathcal{H}-z)^{-1}$, $\im(z)>0$. Then according to \eqref{wig}, the \textit{Stieltjes transform} of $\mathcal{H}$ 
\begin{equation}
s_{N}(z) := N^{-1}\mathrm{Tr}G(z)
\end{equation}
should be close to the Stieltjes transform of the semi-circle:
\begin{equation}
s(z) := \frac{1}{2\pi}\int_{-2}^{2}(x-z)^{-1}\sqrt{4-x^{2}}\,dx
\end{equation}
The local semi-circle law shows that this convergence remains valid at mesoscopic scales $\mathrm{Im}(z) = O(d_{N}^{-1})$ for $1 \ll d_{N} \ll N$. The following is the latest version of this result (in our notation).
\begin{theorem}[Cacciapuoti, Maltsev, Schlein, 2014] \cite[Theorem 1 (i)]{CMS14}
	\label{thm:optimalbound}
	Fix $\tilde{\eta} > 0$ and let $z = t+i\eta$ with $t \in \mathbb{R}$ and $\eta>0$ fixed. Then there are constants $M_0, N_0, C,c, c_0 > 0$ such that
	\begin{equation}
		\mathbb{P}\left(\left|s_{N}(E+z/d_{N})-s(E+z/d_{N})\right| > \frac{Kd_N}{N\eta}\right) \leq (Cq)^{cq^{2}}K^{-q}
	\end{equation}
	for all $\frac{\eta}{d_N} \leq \tilde{\eta}$,$|E + t/d_N| \leq 2 + \eta/d_N$, $N > N_0$ such that $\frac{N\eta}{d_N} \geq M_0$, and  $ q \leq c_0\left(\frac{N\eta}{d_N}\right)^{1/8}$.
\end{theorem}
To prove Theorem \ref{th:maintheorem}, we will start by proving it for the special case of the resolvent $f(x) = (x-z)^{-1}$. Indeed, one can interpret $s_{N}(E+z/d_{N})$ as a random process on the upper-half plane $\mathbb{H}$ and ask whether, after appropriate centering and normalization, a universal limiting process exists. We will show that the function $(N/d_{N})(s_{N}(E+z/d_{N})-\mathbb{E}s_{N}(E+z/d_{N}))$ converges to the $\Gamma'^{+}\textit{-processes}$. These are certain analytic-pathed Gaussian processes defined on $\mathbb{H}$. 

\subsection{Mesoscopic statistics and regularized fractional Brownian motion with $H=0$}
\label{se:mesofrac}
Fractional Brownian motion is a continuous time Gaussian process $B_{H}(t)$ indexed by a number $H \in (0,1)$ and having covariance
\begin{equation}
\mathbb{E}(B_{H}(t)B_{H}(s)) = c_{H}(|t|^{2H}+|s|^{2H}-|t-s|^{2H}) \label{covarfbm}
\end{equation}
where $c_{H}$ is a normalization constant. A generalization of the usual Brownian motion ($H=1/2$), these processes are characterized by their fundamental properties of stationary increments, scale invariance (\textit{i.e.} $B_{H}(at) = a^{H}B_{H}(t)$) and Gaussianity. The parameter $H$ is known as the \textit{Hurst index} and describes the raggedness of the resulting stochastic motion, with the limit of vanishingly small $H$ to be considered the most irregular (see \textit{e.g.} Proposition 2.5 in \cite{Taq2003}). Although the fBm processes were invented by Kolmogorov, they were very widely popularized due to a famous work of Mandelbrot and van Ness \cite{ManvNess68} and since have appeared prominently across mathematics, engineering and finance, among other fields, see \cite{LSSW14} for a survey of fractional Gaussian fields. Until recently however, no relation between the fBm processes and random matrix theory was known. The latter relation (discovered in \cite{FKS13}) goes via the limit $H \to 0$ and the following regularization
\begin{equation}
 B^{(\eta)}_{H}(t) := \frac{1}{2\sqrt{2}}\int_{0}^{\infty}\frac{e^{-\eta s}}{s^{1/2+H}}\left([e^{-its}-1]B_{\mathrm{c}}(ds)+[e^{its}-1]\overline{B_{\mathrm{c}}(ds)}\right) \label{reg}
\end{equation}
where $B_{\mathrm{c}}(s) := B_{1}(s)+iB_{2}(s)$ and $B_1, B_2$ are independent copies of standard Brownian motion. One can verify that in the limit $\eta \to 0$, one recovers precisely the fractional Brownian motion, \textit{i.e.} $B^{(0)}_{H}(t) = B_{H}(t)$. On the other hand, taking instead the limit $H \to 0$ in \ref{reg}, one obtains a process $B^{(\eta)}_{0}(t)$, about which the following was proved:
\begin{theorem}[\cite{FKS13} Fyodorov, Khoruzhenko and Simm]
\label{th:compactconv}
For a GUE random matrix $\mathcal{H}_{\mathrm{GUE}}$, consider the sequence of stochastic processes
\begin{equation}
 W^{(\eta)}_{N}(t) := \log\bigg{|}\det\left(\mathcal{H}_{\mathrm{GUE}}-E-\frac{\tau+i\eta}{d_{N}}\right)\bigg{|}-\log\bigg{|}\det\left(\mathcal{H}_{\mathrm{GUE}}-E-\frac{i\eta}{d_{N}}\right)\bigg{|} \label{guechar}
\end{equation}
and $\tilde{W}^{(\eta)}_{N}(t) :=  W^{(\eta)}_{N}(t) - \mathbb{E}(W^{(\eta)}_{N}(t))$. On any mesoscopic scales of the form $d_{N} \to \infty$ with $d_{N} = o(N/\log(N))$ and with fixed $\tau \in \mathbb{R}$, $\eta>0$ and $E \in (-2,2)$, the process $\tilde{W}^{(\eta)}_{N}$ converges weakly in $L^{2}[a,b]$ to $B^{(\eta)}_{0}$ as $N \to \infty$. 
\end{theorem}

In particular, this gives a functional (in $L^{2}$) version of the CLT of Theorem \ref{th:maintheorem} for GUE random matrices with $f_{k}(x) = \log|x-\tau_{k}-i\eta|-\log|x-i\eta|$. Either by computing the resulting $H^{1/2}$ norm \eqref{covar} or by computing the covariance of $B^{(\eta)}_{0}$ as defined in \eqref{reg}, one finds the logarithmic correlations
\begin{equation}
\mathbb{E}((B^{(\eta)}_{0}(t)-B^{(\eta)}_{0}(s))^{2}) = \frac{1}{2}\log\left(\frac{(t-s)^{2}}{\eta^{2}}+1\right). \label{b0corr}
\end{equation}
Thus $B^{(\eta)}_{0}$ inherits many of the fundamental properties of fBm, including Gaussianity, stationary increments, although now one has the `regularized self-similarity' $B^{(a\eta)}_{0}(at) \overset{d}{=} B^{(\eta)}_{0}(t)$ (the latter following from the scale invariance of the inner product \eqref{covar}). More generally, Gaussian fields with logarithmic correlations have received a great deal of recent attention across mathematics and physics, see \cite{FDR12,FK14} and references therein. The most famous example of such a field is undoubtedly the 2D Gaussian Free Field (GFF) \cite{GFF}, which has important applications in areas such as quantum gravity \cite{DRSV14}, Gaussian multiplicative chaos and Stochastic Loewner Evolution \cite{SS13}. The GFF is also believed to play a central role in random matrix theory. For example, similarly to \eqref{guechar}, it has appeared in relation to the characteristic polynomial, either explicitly \cite{VR97,YHN11} or in what appear to be its various one-dimensional slices \cite{HKOC01,FKS13,W14}. More recently it has appeared as the height function for the minor processes of random matrices \cite{BG13,B14}. 

Going now to the Stieltjes transform $s_{N}(z)$ of Theorem \ref{thm:optimalbound}, the trivial relation $\frac{N}{d_{N}}\re\{s_{N}(\tau+i\eta)\} = \frac{\partial}{\partial \tau}W^{(\eta)}_{N}(\tau)$ suggests that the appropriate limiting object should be related to the derivative of $B^{(\eta)}_{0}(\tau)$. Although such a derivative could obviously be represented by differentiating inside the Fourier integral in \eqref{reg}, it can also be conveniently represented by a random \textit{series}. More generally, for $z \in \mathbb{H}$ and Hurst index $H<1$, define the following `Cayley' series
\begin{equation}
\Gamma'^{+}_{H}(z) := \frac{1}{\sqrt{2}}\left(\frac{z+i}{2}\right)^{2H-2}\sum_{k=0}^{\infty}\sqrt{\frac{\Gamma(2-2H+k)}{\Gamma(2-2H)k!}}\left(\frac{z-i}{z+i}\right)^{k}(\xi^{(1)}_{k}+i\xi^{(2)}_{k}) \label{series}
\end{equation}
where $\{\xi^{(1)}_{k},\xi^{(2)}_{k}\}_{k=0}^{\infty}$ are a family of real \textit{i.i.d.} standard Gaussians. A quick computation with the series \eqref{series} shows that it has zero mean and covariance structure
\begin{equation}
\mathbb{E}(\Gamma'^{+}_{H}(z_1)\overline{\Gamma'^{+}_{H}(z_2)}) = \frac{1}{(i(z_1-\overline{z_2}))^{2-2H}}. \label{cogamma}  
\end{equation}
with $\mathbb{E}(\Gamma'^{+}_{H}(z_1)\Gamma'^{+}_{H}(z_2))=0$. It follows that the processes $\Gamma'^{+}_{H}$ are stationary on horizontal line segments of the complex plane (\textit{c.f.} the stationary increments \eqref{b0corr} for the integrated version). The $\Gamma'^{+}$-processes were originally introduced by Unterberger \cite{U09} in the context of geometric rough path theory and stochastic partial differential equations, but since then the relation to random matrix theory has apparently gone unnoticed. We will show that $\Gamma'^{+}_{0}(z)$ is directly related to a fundamental object of random matrix theory: the normalized trace of the resolvent.

\begin{theorem}
\label{th:vn}
Consider the resolvent $G(z) = (\mathcal{H}-z)^{-1}$. Under the same assumptions as Theorem \ref{th:maintheorem}, the centered and normalized trace
\begin{equation}
 V_{N}(z) := \frac{1}{d_{N}}\left(\tr G(E+z/d_{N})-\mathbb{E}\tr G(E+z/d_{N})\right), \quad \im(z)>0
\end{equation}
converges in the sense of finite-dimensional distributions to $\Gamma'^{+}_{0}(z)$ as $N \to \infty$. That is, for any finite set of points $z_1$,\ldots,$z_M$ in the upper half-plane $\mathbb{H}$, we have
\begin{equation}
(V_{N}(z_1),\ldots,V_{N}(z_M)) \overset{d}{\Longrightarrow} (\Gamma'^{+}_{0}(z_1),\ldots,\Gamma'^{+}_{0}(z_M)), \qquad N \to \infty. \label{vnfindimconv}
\end{equation}
Furthermore, the process $V_{N}$ is tight in the space $\mathcal{U}(D)$ of continuous functions defined on a bounded $N$-independent rectangle $D \subset \mathbb{H}$ and $V_{N}$ converges weakly to $\Gamma'^{+}_{0}$ in $\mathcal{U}(D)$.
\end{theorem}
\begin{proof}
For the finite-dimensional convergence in \eqref{vnfindimconv}, see Section \ref{se:res}. The tightness condition in $\mathcal{U}(D)$ follows from Corollary \ref{cor:arzasc}.
\end{proof}

Intuitively, the underlying reason for the covariance structure \eqref{cogamma} (with $H=0$) appearing in random matrix theory can be traced back to the fundamental relation with the sine-kernel
\begin{equation}
\lim_{\eta_1,\eta_2 \to 0}\lim_{N \to \infty}\mathbb{E}(V_{N}(t_1+i\eta_1)\overline{V_{N}(t_2+i\eta_2)})\bigg|_{d_{N}=N} = \left(\frac{1}{\pi}\frac{\sin(\pi(t_1-t_2))}{(t_1-t_2)}\right)^{2} \label{micro}
\end{equation}
where (heuristically) going to slightly larger scales $d_{N} = N^{\gamma}$ with $0 < \gamma < 1$ has the effect of a large time separation $|t_1-t_2|$ smoothening out the oscillations in the numerator, thus reproducing \eqref{cogamma} with $H=0$ (see \textit{e.g.} \cite{BZ93,P06} for additional heuristics). 

Theorem \ref{th:compactconv} can now be easily extended to Wigner matrices, starting with the identity
\begin{equation}
W^{(\eta)}_{N}(\tau) = \int_{0}^{\tau}\re(V_{N}(t+i\eta))\,dt. \label{charpoly}
\end{equation}
Next, by the rigidity of Theorem \ref{thm:optimalbound}, we have $\mathbb{E}|V_{N}(t+i\eta)|^{2}$ bounded uniformly on compact subsets of $t$ and $\eta \in [\delta,\infty)$ for fixed $\delta>0$ (see Proposition \ref{prop:bulkBound}). Then a standard tightness argument (see \textit{e.g.} \cite{G76}) combined with \eqref{vnfindimconv} allows us to conclude the convergence in distribution as $N \to \infty$,
\begin{equation}
 \int_{0}^{\tau}\re(V_{N}(t+i\eta))\,dt \longrightarrow \int_{0}^{\tau}\re(\Gamma'^{+}_{0}(t+i\eta))\,dt \overset{d}{=} B^{(\eta)}_{0}(\tau).
\end{equation}
This implies that $W^{(\eta)}_{N} \overset{d}{\to} B^{(\eta)}_{0}$, though now in the Wigner case, subject to a more restricted growth of the parameter $d_{N}$ than in Theorem \ref{th:compactconv}. In all cases considered here, optimal conditions on the growth of $d_{N}$ should be anything asymptotically slower than the microscopic scale, \textit{i.e.} we expect our main results to hold provided only that $d_{N} = o(N)$. 

\subsection{Strategy of the proof}
Our proof of Theorem \ref{th:vn} will follow closely the approach popularised by Bai and Silverstein \cite{BS10,BS04}. The technique begins by exploiting the independence of the matrix entries of $\mathcal{H}$ to write $\tr G(z)$ as a sum of martingale differences. Then a classical version of the martingale CLT implies that only $2$ estimates are required in order to conclude asymptotic Gaussianity. For the \textit{macroscopic} regime, this technique was applied successfully to conclude CLTs for many random matrix ensembles, though not without significant computations \cite{BS04,RS06,BWZ09,TV12,BGGM13,RR14}. The mesoscopic regime is characterized by the situation that $\im(z) = O(d_{N}^{-1})$ as $N \to \infty$, which is further problematic in that the majority of bounds for resolvents involve powers of $\im(z)^{-1}$. To overcome this we use the rigid control provided by Theorem \ref{thm:optimalbound} many times, but for technical reasons we were not able to avoid obtaining estimates of order $N^{-1}\im(z)^{-3}$. Such estimates are the source of the restriction on $d_{N}$ in Theorem \ref{th:maintheorem}.

To pass Theorem \ref{th:vn} onto the general linear statistic \eqref{lsmeso} of Theorem \ref{th:maintheorem}, we use an exact formula (see Lemma \ref{lem:extension}):
\begin{equation}
\tilde{X}^{\mathrm{meso}}_{N}(f) = \frac{1}{\pi}\re\int_{0}^{\infty}\int_{-\infty}^{\infty}V_{N}(\tau+i\eta)\overline{\partial} \Psi_{f}(\tau,\eta)\,d\tau\,d\eta \label{genint}
\end{equation}
where $\overline{\partial} := \frac{\partial}{\partial \tau} + i\frac{\partial}{\partial \eta}$ and $\Psi_f$ is a certain 2-dimensional extension of $f$, known as an \textit{almost-analytic extension} \cite{davies1995functional}. Since $\overline{\partial}\Psi_{f}$ is deterministic, we can use our CLT for $V_{N}(\tau+i\eta)$ to conclude a CLT for $\tilde{X}^{\mathrm{meso}}_{N}(f)$. The main problem there is to interchange the distributional convergence for $V_{N}$ with the integrals appearing in \eqref{genint}. To perform such an interchange it will suffice to prove a certain tightness condition which will boil down to having sharp control on $\mathbb{E}|V_{N}(\tau+i\eta)|^{2}$ in the various regimes of $\tau$ and $\eta$. In the bulk of the Wigner semi-circle with $\eta/d_{N} \gg N^{-1}$, the optimal bound of Theorem \ref{thm:optimalbound} plays a key role, since earlier estimates involving $\log(N)$ and $N^{\epsilon}$ factors would lead to a divergent estimate in the mesoscopic regime. In the regions outside the bulk, or with with very small imaginary part $\eta/d_{N} \ll N^{-1}$, we employ the recent variance estimates of \cite{SW13} (see Proposition \ref{prop:sw}) which have the advantage of holding uniformly in $\eta>0$, but the disadvantage of an additional factor $d_{N}^{\epsilon}$ appearing in the bound. In this way we are able to remove the assumption of very rapid decay, which appears in most studies on the mesoscopic regime \cite{S00,BD14,BEYY14}. In contrast, there is no decay requirement in the macroscopic regime and the main important characteristic is the regularity of $f$ \cite{SW13}, while here the decay adds an additional complexity to the problem. It remains an interesting open problem to push our CLT closer to optimal conditions on the decay and regularity of $f$, and on the spectral scale $d_{N}$.

The structure of this paper is as follows. In Section \ref{se:res} we prove the finite-dimensional convergence in Theorem \ref{th:vn} on scales $1 \ll d_{N} \ll N^{1/3}$. In Section \ref{se:testfn} we extend the obtained results to compactly supported functions $f \in C^{1,\alpha}_{\mathrm{c}}(\mathbb{R})$ and show how to replace the assumption of compact support with a suitable decay condition on $f$. Finally, a brief Appendix is included.\\

\textbf{Acknowledgements:} Both authors wish to express thanks to Alice Guionnet for suggesting the main techniques used in the paper. The first author in particular wishes to express his gratitude to Alice Guionnet, who provided helpful advice and support through the NSF grant 6927980 ``\textit{Random Matrices, Free Probability and the enumeration of maps}''. The second author wishes to express his gratitude to Yan Fyodorov, Anna Maltsev and J\'er\'emie Unterberger for stimulating discussions. N. J. Simm was supported on EPSRC grant EP/J002763/1 ``\textit{Insights into Disordered Landscapes via Random Matrix Theory and Statistical Mechanics}''. \\\\

\section{Convergence in law of the Stieltjes transform}
\label{se:res}
The goal of this section is to prove the following:
\begin{proposition}
\label{prop:findim}
Let $z_{1},\ldots,z_{M}$ be $M$ fixed numbers in the upper half of the complex plane $\mathbb{H}$. Under the same assumptions as Theorem \ref{th:vn}, the function $V_{N}$ converges in the sense of finite-dimensional distributions to $\Gamma'^{+}_{0}$, i.e. we have the convergence in law
\begin{equation}
\label{findimconv}
(V_{N}(z_{1}),\ldots,V_{N}(z_{M})) \to (\Gamma'^{+}_{0}(z_{1}),\ldots,\Gamma'^{+}_{0}(z_{M})), \qquad N \to \infty,
\end{equation}
where $\Gamma'^{+}_{0}(z)$ is a Gaussian process on $\mathbb{H}$ with covariance $C(z_1,z_2)$ defined by
\begin{equation}
\mathbb{E}(\Gamma'^{+}_{0}(z_1)\overline{\Gamma'^{+}_{0}(z_2)}) = \frac{1}{(i(z_1-\overline{z_2}))^{2}}
\end{equation}
and
\begin{equation}
\mathbb{E}(\Gamma'^{+}_{0}(z_1)\Gamma'^{+}_{0}(z_2)) = 0
\end{equation}
\end{proposition}
To prove Proposition \ref{prop:findim}, it is enough to fix a linear combination 
\begin{equation}
\mathcal{Z}_{M} := \sum_{p=1}^{M}c_{p}V_{N}(z_{p}) = \sum_{p=1}^{M}c_{p}\frac{1}{d_{N}}(\mathrm{Tr}(G(E+z_{p}/d_{N}))-\mathbb{E}\mathrm{Tr}(G(E+z_{p}/d_{N})))
\end{equation}
and to prove that $\mathcal{Z}_{M}$ converges in distribution to a Gaussian random variable with the appropriate variance. Our starting point is that $\mathcal{Z}_{M}$ can be expressed as a sum of martingale differences, to which a classical version of the martingale CLT can be applied, see Theorem \ref{thm:martclt}. To satisfy the conditions of the martingale CLT we shall follow the technique outlined in Chapter $9$ of the book \cite{BS10} of Bai and Silverstein, which is equivalent to the work \cite{BY05}. Our approach is also valid at the \textit{macroscopic} scales considered in \cite{BY05} and we feel gives a somewhat more accessible proof in this case.

\subsection{Method of martingales and some preliminaries}
Here we outline the martingale method and provide the notation used in the remainder of this Section. Let $\mathbb{E}_{k}$ denote the conditional expectation with respect to the $\sigma$-algebra generated by the upper-left $k \times k$ corner of the Wigner matrix $W$. Then we have the martingale decomposition
\begin{equation}
\mathcal{Z}_{M} = \sum_{k=1}^{N}X_{k,N} \label{martdecomp}
\end{equation}
where
\begin{equation}
X_{k,N} := (\mathbb{E}_{k}-\mathbb{E}_{k-1})\sum_{p=1}^{M}c_{p}\frac{1}{d_{N}}\mathrm{Tr}(G(E+z_{p}/d_{N}))
\end{equation}
Therefore, to prove Proposition \ref{prop:findim}, it will suffice to check the following two conditions:
\begin{enumerate}
	\item \textit{The Lindeberg condition:} for all $\epsilon>0$, we have
		\begin{equation}
			\sum_{k=1}^N \mathbb{E}(|X_{k,N}|^{2} \mathbf{1}_{|X_{k,N}|>\epsilon})\to 0, \quad N \to \infty. \label{lindeberg}
		\end{equation}
	\item \textit{Conditional variance:} we have the convergence in probability
\begin{align}
	&\sum_{k=1}^N \mathbb{E}_{k-1}[|X_{k,N}|^{2}] \rightarrow \sum_{l,m = 1}^M c_l \overline{c_m}C(z_l,\overline{z_m}), \quad N \to \infty, \label{covarconv}\\
	&\sum_{k=1}^N \mathbb{E}_{k-1}[X_{k,N}^{2}] \rightarrow \sum_{l,m = 1}^M c_l c_m C(z_l,z_m), \quad N \to \infty, \label{covarconv2}
	\end{align}
\end{enumerate}
where $C(z_l,z_m) = (i(\tau_l-\tau_m+i(\eta_l+\eta_m)))^{-2}$ denotes the covariance in Proposition \ref{prop:findim}.\\

Before we proceed with the proof of these conditions, we provide some of the relevant notation.\\

\textbf{\textit{Important notation:}} Until now the complex numbers $z_{p}$ were independent of $N$. For notational convenience and in the remainder of this Section only, we will now allow the implicit $N$-dependence
\begin{equation}
z_{p} := E+\frac{\tau_{p}+i\eta_{p}}{d_{N}} \label{zn}
\end{equation}
As before, the sequence $d_{N} \to \infty$ as $N \to \infty$ with $d_{N}/N \to 0$ and $\tau_{p}, \eta_{p}$ are fixed real numbers with $\eta_{p} \neq 0$. We fix $E \in (-2+\delta,2-\delta)$ strictly inside the support of the limiting semi-circle for some small $\delta>0$. 

Let $\mathcal{H}_{k}$ be the $N-1 \times N-1$ Wigner matrix obtained by erasing the $k^{\mathrm{th}}$ row and column from $\mathcal{H}$. We denote by $G_{k}(z) = (\mathcal{H}_{k}-z)^{-1}$ the corresponding resolvent. The following formula, a consequence of the \textit{Schur complement formula} from Linear Algebra, will play an important role:
\begin{equation}
\tr(G(z))-\tr(G_{k}(z)) = \frac{1+h_{k}^{\dagger}G_{k}(z)^{2}h_{k}}{\mathcal{H}_{kk}-z-h_{k}^{\dagger}G_{k}(z)h_{k}}
\end{equation}
where $h_{k}$ is the $k^{\mathrm{th}}$ column of $\mathcal{H}$ with the $k^{\mathrm{th}}$ entry removed. 

Recall the following standard notation for convergence of random variables in $L^{p}$. For a sequence of random variables $\{X_{N}\}_{N=1}^{\infty}$, we write $X_{N} = O_{L^{p}}(u(N))$ to mean there exists a constant $c$ such that $\mathbb{E}|X_{N}|^{p} \leq cu(N)$ for all $N$ large enough. We will repeatedly use the standard fact that if $X_{N}$ converges to $X$ in probability and $Y_{N}$ converges to zero in $L^{p}$, $p\geq 1$, then $X_{N}+Y_{N}$ converges to $X$ in probability.

We start with the proof of the Lindeberg condition \eqref{lindeberg} which follows from the following stronger result (due to the trivial inequality $|X_{k,N}|^2\mathbf{1}_{|X_{k,N}|>\epsilon} \leq \epsilon^{2-\delta}|X_{k,N}|^\delta$ with $\delta>2$):

\begin{lemma}[Lyapunov]
For all mesoscopic scales $1 \ll d_{N} \ll N^{1-\epsilon}$ with $\epsilon>0$, there is an integer $\delta>2$ such that
\begin{equation}
\sum_{k=1}^{N}\mathbb{E}\bigg{|}(\mathbb{E}_{k}-\mathbb{E}_{k-1})\sum_{p=1}^{M}c_{p}\frac{1}{d_{N}}\tr G(z) \bigg{|}^{\delta} \to 0, \qquad N \to \infty.
\end{equation}
\end{lemma}

\begin{proof}
By the triangle inequality it suffices to verify the claim when $M=1$, $c_{1}=1$. By definition of $\mathbb{E}_{k}$ we have $(\mathbb{E}_{k}-\mathbb{E}_{k-1})d_{N}^{-1}\mathrm{Tr}G(z) = (\mathbb{E}_{k}-\mathbb{E}_{k-1})Z_{k,N}$ where $Z_{k,N} := d_{N}^{-1}(\mathrm{Tr}G(z)-\mathrm{Tr}G_{k}(z))$. Then Schur's complement formula implies
\begin{align}
Z_{k,N} &= \frac{1}{d_{N}}\frac{1+h_{k}^{\dagger}G_{k}(z_{1})^{2}h_{k}}{\mathcal{H}_{kk}-z_{1}-h_{k}^{\dagger}G_{k}(z_{1})h_{k}}\\
&= \frac{1}{d_{N}}(1+\delta^{k,2}_{N}(z_{1})+N^{-1}\tr(G_{k}(z_{1})^{2})G_{kk}(z_{1})
\end{align}
where we made use of the identity for the diagonal elements of the resolvent
\begin{equation}
G_{kk}(z_{1}) = \frac{1}{\mathcal{H}_{kk}-z_{1}-h_{k}^{\dagger}G_{k}(z_{1})h_{k}}
\end{equation}
and defined
\begin{equation}
\delta^{k,n}_{N}(z_{1}) := h_{k}^{\dagger}G_{k}(z_{1})^{n}h_{k}-N^{-1}\tr(G_{k}(z_{1})^{n}) \label{deltakn}
\end{equation}
By the conditional Jensen inequality, we have $|\mathbb{E}_{k}Z_{k}|^{\delta} \leq \mathbb{E}_{k}|Z_{k}|^{\delta}$. Hence it is sufficient to prove 
\begin{equation}
\sum_{k=1}^{N}\mathbb{E}|Z_{k,N}|^{\delta} \to 0, \qquad N \to \infty \label{zkn}
\end{equation}
The limit \eqref{zkn} follows from standard concentration inequalities applied to the variables $d_{N}^{-1}\delta^{k,2}_{N}(z)$, $G_{kk}(z)$ and $d_{N}^{-1}N^{-1}\tr G(z)^{2}$. In particular, Lemmas \ref{d1bound}, \ref{lem:derivbound} and \ref{le:offdiag} show that for any fixed $q>0$, we have the estimates
\begin{align}
&d_{N}^{-1}\delta^{k,2}_{N}(z) = O_{L^{q}}((d_{N}/N)^{q/2}),\\
&G_{kk}(z) = O_{L^{q}}(1),\\
&d_{N}^{-1} N^{-1} \tr G(z)^{2} \leq O_{L^{q}}\left(\max\{d_{N}^{-q},(d_{N}/N)^{q}\}\right), \label{globalprob}
\end{align}
Then applying Cauchy-Schwarz and choosing $\delta>0$ large enough, we obtain \eqref{zkn}.
\end{proof}

\begin{remark}
In the macroscopic regime $d_{N}=1$, one can argue similarly that $Z_{k,N} = (1+s'(z))(-z-s(z))^{-1} + O(N^{-1/2})$ with high probability. The leading term in this asymptotic is deterministic and does not contribute to $(\mathbb{E}_{k}-\mathbb{E}_{k-1})Z_{k,N}$, while the error term is small enough to imply \eqref{zkn}.
\end{remark}

We now proceed to the remaining and most challenging part of the proof of Proposition \ref{prop:findim}, which is to verify condition \eqref{covarconv}. Before we proceed, it's worth noting that both $X_{k}^{2}$ and $|X_{k}|^{2}$ are finite linear combinations of terms of the form
\begin{equation}
(\mathbb{E}_{k}-\mathbb{E}_{k-1})\frac{1}{d_{N}}\mathrm{Tr}G(z_{1})\times (\mathbb{E}_{k}-\mathbb{E}_{k-1})\frac{1}{d_{N}}\mathrm{Tr}G(z_{2})
\end{equation}
and so it suffices to prove the convergence for a single mixed term in the linear combination. Setting
\begin{equation}
Y_{k}(z) := (\mathbb{E}_{k}-\mathbb{E}_{k-1})d_{N}^{-1}\mathrm{Tr}G(z),
\end{equation}
our essential goal in the remainder of this section will be to prove that we have the convergence in probability
\begin{equation}
C_{N}(z_{1},\overline{z_{2}}) := \sum_{k=1}^{N}\mathbb{E}_{k-1}[Y_{k}(z_{1})Y_{k}(\overline{z_{2}})] \to \frac{1}{(i(\tau_1-\tau_2+i(\eta_1+\eta_2)))^{2}}, \qquad N \to \infty \label{convgamma}.
\end{equation}
In what follows, the proof of \eqref{convgamma} is divided into $3$ main subsections: in section \ref{se:simp} we rewrite $C_{N}(z_{1},z_{2})$ in a form suitable for the computation of asymptotics, then in section \ref{se:comp} the main asymptotic results are obtained and finally in section \ref{se:conv} they are used to prove \eqref{convgamma}.
\subsection{Simplifying the covariance kernel}
\label{se:simp}
Our first Proposition shows that $C_{N}(z_{1},z_{2})$ can be approximated in the following way

\begin{proposition}
In terms of the variables \eqref{deltakn}, define the covariance kernel
\begin{equation}
\tilde{C}_{N}(z_{1},z_{2}) := \frac{1}{d_{N}^{2}}\frac{\partial^{2}}{\partial z_{1} \partial z_{2}}\left[s(z_{1})s(z_{2})\sum_{k=1}^{N}\mathbb{E}_{k-1}[\mathbb{E}_{k}\delta^{k,1}_{N}(z_{1})\mathbb{E}_{k}\delta^{k,1}_{N}(z_{2})]\right] \label{apcoker}
\end{equation}
Then we have
\begin{equation}
C_{N}(z_{1},z_{2}) = \tilde{C}_{N}(z_{1},z_{2}) + O_{L^1}(\sqrt{d_{N}^{2}/N}) \label{covarest}
\end{equation}
\end{proposition}

\begin{proof}
As in the proof of the Lindeberg condition, we start with Schur's complement formula which implies that
\begin{equation}
Y_{k}(z) = (\mathbb{E}_{k}-\mathbb{E}_{k-1})\frac{1}{d_{N}}\frac{1+h_{k}^{\dagger}G_{k}(z)^{2}h_{k}}{\mathcal{H}_{kk}-z-h_{k}^{\dagger}G_{k}(z)h_{k}}
\end{equation}
Rewriting $Y_{k}(z)$ via the small terms \eqref{deltakn} and expanding, we obtain the exact identity
\begin{equation}
\label{errorsgkmt}
Y_{k}(z) = (\mathbb{E}_{k}-\mathbb{E}_{k-1})\frac{\partial}{\partial z}\frac{1}{d_{N}}\frac{\mathcal{H}_{kk}-\delta^{k,1}_{N}(z)}{z+N^{-1}\tr(G_{k}(z))} + \epsilon_{k,N}(z)
\end{equation}
where
\begin{equation}
\epsilon_{k,N}(z) := (\mathbb{E}_{k}-\mathbb{E}_{k-1})\left(\frac{1}{d_{N}}\frac{(\mathcal{H}_{kk}-\delta^{k,1}_{N}(z))^{2}G_{kk}(z)}{(z+N^{-1}\tr(G_{k}(z)))^{2}}-\frac{1}{d_{N}}\frac{\delta^{k,2}_{N}(\mathcal{H}_{kk}-\delta^{k,1}_{N}(z))}{(z+N^{-1}\tr(G_{k}(z)))}\right).
\end{equation}
This identity is implicit in the work \cite{BY05} (see Section 4.1 in \cite{BY05}), but we provide the derivation in the Appendix, Lemma \ref{le:schurident}. Then as in the proof of \eqref{zkn}, we see that $\epsilon_{k,N}(z) = O_{L^1}(d_{N}/N)$ uniformly in $k$. Similarly, by Lemma \ref{d1bound} we can replace $(z+N^{-1}\tr (G_{k}(z)))^{-1}$ with $-s(z)$, costing an error of the same order. Therefore, we have $Y_{k}(z) = \tilde{Y}_{k}(z) + O_{L^1}(d_{N}/N)$ where
\begin{equation}
\tilde{Y}_{k}(z) = -(\mathbb{E}_{k}-\mathbb{E}_{k-1})\frac{\partial}{\partial z}\frac{1}{d_{N}}s(z)(\mathcal{H}_{kk}-\delta^{k,1}_{N}(z))
\end{equation}
Using properties of the conditional expectation, we compute that
\begin{equation}
	\sum_{k=1}^{N}\mathbb{E}_{k-1}[\tilde{Y}_{k}(z_{1})\tilde{Y}_{k}(z_{2})] = \frac{1}{d_{N}^{2}}\frac{\partial^2}{\partial z_1 \partial z_2}s(z_1)s(z_2)+\tilde{C}_{N}(z_{1},z_{2})
\end{equation}
Then the covariance $C_{N}(z_{1},z_{2})$ can be estimated as
\begin{equation}
\begin{split}
C_{N}(z_{1},z_{2}) &= \tilde{C}_{N}(z_{1},z_{2}) + \sum_{k=1}^{N}\mathbb{E}_{k-1}(Y_{k}(z_{1})-\tilde{Y}_{k}(z_{1}))\tilde{Y}_{k}(z_{2})+\sum_{k=1}^{N}\mathbb{E}_{k-1}(Y_{k}(z_{2})-\tilde{Y}_{k}(z_{2}))\tilde{Y}_{k}(z_{1})\\
&+\sum_{k=1}^{N}\mathbb{E}_{k-1}(Y_{k}(z_1)-\tilde{Y}_{k}(z_1))(Y_{k}(z_2)-\tilde{Y}_{k}(z_2)) + O_{L^1}(1/d_N^2)
\end{split}
\end{equation}
where we used that $s'(z_1)s'(z_2)$ are uniformly bounded for any fixed $E \in (-2,2)$. By our estimates for $Y_{k}-\tilde{Y}_{k}$, the last term above is $O_{L^1}(d_{N}^{2}/N)$. For the middle terms, we apply Cauchy-Schwarz (twice) to obtain
\begin{equation}
\begin{split}
&\mathbb{E}\bigg{|}\sum_{k=1}^{N}\mathbb{E}_{k-1}\tilde{Y}_{k}(z_{1})(Y_{k}(z_{2})-\tilde{Y}_{k}(z_{2}))\bigg{|} \leq \mathbb{E}\sum_{k=1}^{N}|Y_{k}(z_{1})||Y_{k}(z_{2})-\tilde{Y}_{k}(z_{2})|\\
&\leq \sqrt{\sum_{k=1}^{N}\mathbb{E}|Y_{k}(z_1)-\tilde{Y}_{k}(z_1)|^{2}\sum_{k=1}^{N}\mathbb{E}|\tilde{Y}_{k}(z_2)|^{2}}\\
&=\sqrt{\sum_{k=1}^{N}\mathbb{E}|Y_{k}(z_1)-\tilde{Y}_{k}(z_1)|^{2}}\sqrt{\mathbb{E}\tilde{C}_{N}(z_2,\overline{z_2})+d_{N}^{-2}|s'(z_2)|^{2}}
\end{split}
\end{equation}
The first term in the product above is $O(\sqrt{d_{N}^{2}/N})$. In the remainder of this section, it will become clear that $\tilde{C}_{N}(z_1,z_2)$ is bounded in $L^1$, see Proposition \ref{prop:cn}. 
\end{proof}

\begin{remark}
\label{rem:cauchy}
Before we proceed further, note that the approximate covariance kernel $\tilde{C}_{N}(z_{1},z_{2})$ in \eqref{apcoker} is naturally expressed in terms of the auxiliary kernel
\begin{equation}
K_{N}(z_{1},z_{2}) := \sum_{k=1}^{N}\mathbb{E}_{k-1}[\mathbb{E}_{k}\delta^{k,1}_{N}(z_{1})\mathbb{E}_{k}\delta^{k,1}_{N}(z_{2})].
\end{equation}
Then by Cauchy's integral formula and analyticity, we can write the derivatives in \eqref{apcoker} as 
\begin{equation}
\tilde{C}_{N}(z_{1},z_{2}) := \frac{1}{d_{N}^{2}}\frac{1}{(2\pi i)^{2}}\oint_{\mathcal{S}_{z_{1}}}d\omega_1 \oint_{\mathcal{S}_{z_{2}}}d\omega_2 \frac{s(\omega_1)s(\omega_2)}{(z_{1}-\omega_1)^{2}(z_{2}-\omega_2)^{2}}\,K_{N}(\omega_1,\omega_2) \label{cauchyker}
\end{equation}
where $\mathcal{S}_{z}$ is a small circle with center $z$ and radius $1/(2d_{N})\sqrt{(\tau_1-\tau_2)^{2}+(\eta_1-\eta_2)^{2}}$. This ensures that for fixed $\tau_1$,$\tau_2$,$\eta_1$,$\eta_2$ and $N$ large enough, $\mathcal{S}_{z_{1}}$ and $\mathcal{S}_{z_{2}}$ are disjoint sets. In the degenerate case that $z_{1}=z_{2}$, it is enough to use the Cauchy integral formula with a single circle $\mathcal{S}_{z_{1}}$.

If we obtain uniform estimates on $K_{N}$ of the form $K_{N} = \tilde{K}_{N}+O_{L^1}(u(N))$, then the error in approximating $\tilde{C}_{N}(z_{1},z_{2})$ is of the same order in $N$:
\begin{equation}
\begin{split}
&\mathbb{E}\bigg{|}\frac{1}{d_{N}^{2}}\frac{1}{(2\pi i)^{2}}\oint_{\mathcal{S}_{z_{1}}}d\omega_1 \oint_{\mathcal{S}_{z_{2}}}d\omega_2 \frac{s(\omega_1)s(\omega_2)}{(z_{1}-\omega_1)^{2}(z_{2}-\omega_2)^{2}}\,|K_{N}-\tilde{K}_{N}|\bigg{|}\\
&\leq \frac{1}{d_{N}^{2}}\frac{1}{4\pi^{2}}\oint_{\mathcal{S}_{z_{1}}}d\omega_1 \oint_{\mathcal{S}_{z_{2}}}d\omega_2 \frac{|s(\omega_1)||s(\omega_2)|}{|z_{1}-\omega_1|^{2}|z_{2}-\omega_2|^{2}}\,|u(N)|\\
&\leq \frac{1}{16 \pi^{2}}\frac{1}{(\tau_1-\tau_2)^{2}+(\eta_1-\eta_2)^{2}}\, |u(N)|
\end{split}
\end{equation}
The conclusion of this remark is that it will be sufficient just to understand the convergence in $L^1$ of the kernel $K_{N}(z_{1},z_{2})$. 
\end{remark}

Our first Lemma in this direction rewrites $K_{N}(z_{1},z_{2})$ in terms of the matrix elements of the resolvent $G_{k}(z) := (\mathcal{H}_{k}-z)^{-1}$. We will frequently make use of the shorthand notation $G^{(p)}_{k}:= G_{k}(z_{p})$, $p=1,2$ to emphasize the dependence on the variables $z_{1}$ and $z_{2}$.
\begin{lemma}
The covariance kernel $K_{N}(z_{1},z_{2})$ satisfies the exact identity
\begin{align}
K_{N}(z_{1},z_{2}) =  &N^{-2}\sum_{k=1}^{N}\mathbb{E}_{k-1}\sum_{i<k,j<k}\mathbb{E}_{k}(G^{(1)}_{k})_{ij}\mathbb{E}_{k}(G^{(2)}_{k})_{ji} \label{s1k}\\
&+N^{-2}\sum_{k=1}^{N}\mathbb{E}_{k-1}\sum_{i<k}\mathbb{E}_{k}(G^{(1)}_{k})_{ii}\mathbb{E}_{k}(G^{(2)}_{k})_{ii}\beta_{ik} \label{s2}
\end{align}
where $\beta_{ik}$ is expressed in terms of the fourth moments
\begin{equation}
\beta_{ik} := \mathbb{E}(|W_{ik}|^{2}-1)^{2}
\end{equation}
\end{lemma}

\begin{proof}
By definition we have
\begin{align}
\mathbb{E}_{k}(\delta^{k,1}_{N}(z_{1})) &=N^{-1}\sum_{i<k,j<k}\mathbb{E}_{k}(G^{(1)}_{k})_{ij}\overline{W_{ik}}W_{jk}-N^{-1}\sum_{j<k}\mathbb{E}_{k}(G^{(1)}_{k})_{jj}\\
&= N^{-1}\sum_{\substack{i<k,j<k\\i \neq j}}\mathbb{E}_{k}(G^{(1)}_{k})_{ij}\overline{W_{ik}}W_{jk}+N^{-1}\sum_{j<k}\mathbb{E}_{k-1}(G^{(1)}_{k})_{jj}(|W_{jk}|^{2}-1)
\end{align}
Multiplying out the resulting terms, we obtain
\begin{align}
&\mathbb{E}_{k-1}[\mathbb{E}_{k}(\delta^{1,k}_{N}(z))\mathbb{E}_{k}(\delta^{1,k}_{N}(z))]\\
&=N^{-2}\mathbb{E}_{k-1}\sum_{\substack{i<k,j<k\\p<k,q<k\\i \neq j, p \neq q}}\mathbb{E}_{k}(G^{(1)}_{k})_{ij}\mathbb{E}_{k}(G^{(2)}_{k})_{pq}\overline{W_{ik}}W_{jk}\overline{W_{pk}}W_{qk} \label{s1term}\\
&+N^{-2}\mathbb{E}_{k-1}\sum_{\substack{i<k,j<k\\i \neq j}}\mathbb{E}_{k}(G^{(1)}_{k})_{ij}\overline{W_{ik}}W_{jk}\sum_{j<k}\mathbb{E}_{k}(G^{(2)}_{k})_{jj}(|W_{jk}|^{2}-1) \label{midterm1}\\
&+N^{-2}\mathbb{E}_{k-1}\sum_{\substack{i<k, j<k\\i \neq j}}\mathbb{E}_{k}(G^{(2)}_{k})_{ij}\overline{W_{ik}}W_{jk}\sum_{j<k}\mathbb{E}_{k}(G^{(1)}_{k})_{jj}(|W_{jk}|^{2}-1) \label{midterm2}\\
&+N^{-2}\mathbb{E}_{k-1}\sum_{i<k,j<k}\mathbb{E}_{k}(G^{(1)}_{k})_{ii}\mathbb{E}_{k}(G^{(2)}_{k})_{jj}(|W_{jk}|^{2}-1)(|W_{ik}|^{2}-1) \label{s2term}
\end{align}
It is clear that the sums \eqref{midterm1} and \eqref{midterm2} are identically zero, since the vector $\{W_{ik}\}_{i\neq k}$ consists of centered independent random variables satisfying $\mathbb{E}|W_{ik}|^{2}=1$. Similarly, the first summation \eqref{s1term} will be zero unless $i=q$ and $p=j$, this gives the first sum on the right-hand side of \eqref{s1k}. The last term \eqref{s2term} will be zero unless $i=j$, which gives the second sum in \eqref{s2}.
\end{proof}

We now proceed with the estimation of the two sums in \eqref{s1k} and \eqref{s2}. To do this we need precise estimates on the resolvent matrix elements appearing in the sums.
\begin{lemma}[Bound on the resolvent]
\label{le:offdiag}
Let $z = E+\frac{\tau+i\eta}{d_{N}}$ as in \eqref{zn} and consider an off-diagonal resolvent matrix element $(G_{k}(z))_{pq}$ with $p \neq q$. Then for any positive integer $s$ we have positive constants $c,C$ such that
\begin{equation}
\mathbb{E}|(G_{k}(z))_{pq}|^{s} \leq (Cs)^{cs}\left(\frac{d_{N}}{\eta N}\right)^{s/2} \label{gpq}
\end{equation}
for all $k,p,q$, $E, \tau \in \mathbb{R}$, $\eta>0$ and $d_{N}>0$. For the diagonal matrix elements, the same result holds but with $(G_{k}(z))_{pp} - s(z)$ in place of $(G_{k}(z))_{pq}$.
\end{lemma}

\begin{proof}
This follows from Lemma 5.3 in \cite{CMS14}, see also previous works on the local semi-circle laws \cite{ESY08,ESY09,EYY12,EYY12rig}.
\end{proof}

The second sum \eqref{s2} over diagonal elements of the resolvent can now be dispensed with immediately:
\begin{lemma}
Assume that
\begin{equation}
\sup_{N,i,k>0}\mathbb{E}(|W_{ik}|^{2}-1)^{2} < \infty
\end{equation}
Then for all points $z_{1}$ and $z_{2}$ with non-zero imaginary part and on all mesoscopic scales $1 \ll d_{N} \ll N$, we have the convergence in $L^{1}$:
\begin{equation}
\label{stat_4th_mom_est}
\frac{1}{d_{N}^{2}}\frac{\partial^{2}}{\partial z_{1} \partial z_{2}}\left[s(z_{1})s(z_{2})N^{-2}\sum_{k=1}^{N}\sum_{i<k}\mathbb{E}_{k}(G^{(1)}_{k})_{ii}\mathbb{E}_{k}(G^{(2)}_{k})_{ii}\beta_{ik}\right] \to 0, \qquad N \to \infty
\end{equation}
\end{lemma}

\begin{proof}
By Lemma \ref{le:offdiag} we have
\begin{align}
&N^{-2}\sum_{k=1}^{N}\sum_{i<k}\mathbb{E}_{k}(G^{(1)}_{k})_{ii}\mathbb{E}_{k}(G^{(2)}_{k})_{ii}\beta_{ik}  = N^{-2}s(z_{1})s(z_{2})\sum_{k=1}^{N}\sum_{i<k}\beta_{ik}+O_{L^1}(\sqrt{d_{N}/N}) \label{4th_mom_est}
\end{align}
where we assumed that the fourth moment of each matrix entry $W_{ik}$ is finite. Since the left-hand side of \eqref{4th_mom_est} is analytic, we can use the strategy outlined in Remark \ref{rem:cauchy}. Hence after inserting \eqref{4th_mom_est} into the left-hand side of \eqref{stat_4th_mom_est}, the $O_{L^1}(\sqrt{d_{N}/N})$ bound remains of the same order while the leading term is of order $d_{N}^{-2}$ since $s$ and $s'$ are analytic and uniformly bounded provided $E \in (-2+\delta,2-\delta)$. This completes the proof of the Lemma.
\end{proof}

\begin{remark}This shows that in the mesoscopic regime, the process $V_{N}(z)$ is \textit{insensitive} to the value of the fourth cumulants of the matrix elements, provided they are finite. This is in contrast to the regime of global fluctuations where the fourth cumulant is known to appear explicitly in the limiting covariance formula, see \textit{e.g.} \cite{BY05,Sh11}. This suggests that the Gaussian fluctuations obtained in the mesoscopic regime are more universal than in the global regime.
\end{remark}

The remaining challenge now is to compute the first sum in \eqref{s1k}. For that we shall show that to leading order on the scales $1 \ll d_{N} \ll N^{1/3}$, the sum
\begin{equation}
\mathcal{S}_{1,k}(z_{1},z_{2}) := \frac{1}{N}\sum_{p<k,q<k}\mathbb{E}_{k}(G_{k}(z_{1}))_{qp}\mathbb{E}_{k}(G_{k}(z_{2}))_{pq} \label{s1}
\end{equation}
satisfies a self-consistent equation. The proof relies on heavy computations.

\subsection{Calculation of the sum $\mathcal{S}_{1,k}(z_{1},z_{2})$}
\label{se:comp}
The goal of this section will be to prove the following
\begin{proposition}
\label{prop:s1}
Consider the quantity $\mathcal{S}_{1,k}(z_{1},z_{2})$ in \eqref{s1}. We have the estimate
\begin{equation}
z_{1}\mathcal{S}_{1,k}(z_{1},z_{2})=-s(z_{1})\mathcal{S}_{1,k}(z_{1},z_{2})-\frac{N-k}{N}s(z_{2})-\frac{N-k}{N}s(z_{2})\mathcal{S}_{1,k}(z_{1},z_{2})+O_{L^1}((d_{N}^{3}N^{-1})^{1/2}) \label{recs1}
\end{equation}
where the $O_{L^1}$ bound is uniform in $k=1,\ldots,N$.
\end{proposition}
In what follows, the proof of this Proposition will be divided into a number of Lemmas which eventually culminate in Lemma \eqref{le:s12}. We start with the Green function perturbation identity
\begin{equation}
\label{pertident}
z_{1}G_{k}(z_{1}) = -I_{N-1}+N^{-1/2}\sum_{\substack{i \neq k\\ j \neq k}}W_{ij}e_{i}e_{j}'G_{k}(z_{1})
\end{equation}
where the vector $e_{i}$ has entries $(e_{i})_{j} = \delta_{i,j}$ for $j<k$ and $(e_{i})_{j} = \delta_{i,j-1}$. It turns out to be helpful to separate out the correlations between $W_{ij}$ and $G_{k}$ by introducing the matrix $G_{kij}$ defined as the resolvent of the matrix $\mathcal{H}_{k}$ with entries $(i,j)$ and $(j,i)$ replaced with $0$. Simple algebra shows this is a perturbation of the original resolvent:
\begin{equation}
\label{gkijpert}
G^{(1)}_{k}-G^{(1)}_{kij} = -N^{-1/2}G^{(1)}_{kij}c_{ij}(W_{ij}e_{i}e_{j}'+W_{ji}e_{j}e_{i}')G^{(1)}_{k}.
\end{equation}
where $c_{ij}=1$ if $i \neq j$ and $c_{ii}=1/2$. We now insert \eqref{gkijpert} directly into \eqref{pertident} leading to the following expansion.
\begin{lemma}
\label{le:gexpan}
The matrix elements of the resolvent $(G^{(1)}_{k})_{pq}$ satisfy
\begin{equation}
\begin{split}
&z_{1}(G^{(1)}_{k})_{pq} = -\delta_{pq}+N^{-1/2}\sum_{j \neq k}W_{pj}(G^{(1)}_{kpj})_{jq}-s(z_{1})(1-3/(2N))(G^{(1)}_{k})_{pq}\\
&-N^{-1}\sum_{j \neq k}c_{pj}|W_{pj}|^{2}((G^{(1)}_{kpj})_{jj}-s(z_{1}))(G^{(1)}_{k})_{pq}-N^{-1}\sum_{j \neq k}c_{pj}(|W_{pj}|^{2}-1)s(z_{1})(G^{(1)}_{k})_{pq}\\
&-N^{-1}\sum_{j \neq k}c_{pj}W_{pj}^{2}(G^{(1)}_{kpj})_{jp}(G^{(1)}_{k})_{jq} \label{expanlem}
\end{split}
\end{equation}
\end{lemma}

Now inserting \eqref{expanlem} into the definition \eqref{s1}, we obtain the decomposition
\begin{equation}
z_{1}\mathcal{S}_{1}(z_{1},z_{2})(k) = \mathcal{S}_{1,1}+\mathcal{S}_{1,2}+\mathcal{S}_{1,3}+\mathcal{S}_{1,4}+\mathcal{S}_{1,5}+\mathcal{S}_{1,6}
\end{equation}
where
\begin{align}
\mathcal{S}_{1,1} &= -N^{-1}\sum_{p<k}\mathbb{E}_{k}(G^{(2)}_{k})_{pp} \label{s11}\\
\mathcal{S}_{1,2} &= N^{-3/2}\sum_{\substack{j \neq k\\p<k,q<k}}W_{pj}\mathbb{E}_{k}(G^{(1)}_{kpj})_{jq}\mathbb{E}_{k}(G^{(2)}_{k})_{qp}
\label{s12}\\
\mathcal{S}_{1,3}&=-s(z_{1})(1-3/(2N))\mathcal{S}_{1} \label{s13}\\
\mathcal{S}_{1,4}&=-N^{-2}\sum_{\substack{j \neq k\\p<k,q<k}}c_{pj}\mathbb{E}_{k}|W_{pj}|^{2}((G^{(1)}_{kpj})_{jj}-s(z_{1}))(G^{(1)}_{k})_{pq}\mathbb{E}_{k}(G^{(2)}_{k})_{qp} \label{s14}\\
\mathcal{S}_{1,5}&=-N^{-2}\sum_{\substack{j \neq k\\ p<k,q<k}}c_{pj}\mathbb{E}_{k}(|W_{pj}|^{2}-1)s(z_{1})(G^{(1)}_{k})_{pq}\mathbb{E}_{k}(G^{(2)}_{k})_{qp} \label{s15}\\
\mathcal{S}_{1,6}&=-N^{-2}\sum_{\substack{j \neq k\\ p<k,q<k}}c_{pj}\mathbb{E}_{k}W_{pj}^{2}(G^{(1)}_{kpj})_{jp}(G^{(1)}_{k})_{jq}\mathbb{E}_{k}(G^{(2)}_{k})_{qp} \label{s16}
\end{align}

To estimate these sums we will apply Lemma \ref{le:offdiag} repeatedly, together with the following Lemma which allows us to bound the matrix elements of the perturbation $(G_{kij}(z))_{pq}$.
\begin{lemma}[Bound on the perturbation]
\label{le:gkij}
Provided that $W$ has finite moments of order $4s$, the same bound \eqref{gpq} of Lemma \ref{le:offdiag} holds with $(G_{kij}(z))_{pq}$ in place of $(G_{k}(z))_{pq}$.
\end{lemma}
\begin{proof}
Using the perturbation identity \eqref{gkijpert} we get
\begin{equation}
(G_{kij})_{pq} = (G_{k})_{pq}+c_{ij}W_{ij}N^{-1/2}(G_{kij})_{pi}(G_{k})_{jq}+c_{ij}W_{ji}N^{-1/2}(G_{kij})_{pj}(G_{k})_{iq} \label{melpert}
\end{equation}
Iterating \eqref{melpert} once more and applying Minkowski's inequality shows that
\begin{align*}
|(G_{kij})_{pq}|^{s} &\leq c_{s}|(G_{k})_{pq}|^{s}+c_{s}|W_{ij}|^{s}N^{-s/2}\left(|(G_{k})_{pi}(G_{k})_{jq}|^{s}+|(G_{k})_{pj}(G_{k})_{iq}|^{s}\right)\\
&+c_{s}|W_{ij}|^{2s}N^{-s}\left(|(G_{kij})_{pi}(G_{k})_{ji}(G_{k})_{jq}|^{s}+|(G_{kij})_{pj}(G_{k})_{ii}(G_{k})_{jq}|^{s}\right.\\
&\left.+|(G_{kij})_{pi}(G_{k})_{jj}(G_{k})_{iq}|^{s}+|(G_{kij})_{pj}(G_{k})_{ij}(G_{k})_{iq}|^{s}\right)
\end{align*}
for some constant $c_{s}$ depending only on $s$. From the deterministic bound $|(G_{kij})_{pq}| \leq \norm{G_{kij}} \leq d_{N}/\eta$ and \eqref{gpq}, we find by Cauchy-Schwarz that 
\begin{equation}
\mathbb{E}|(G_{kij})_{pq}|^{s} \leq c_{s}\mathbb{E}|(G_{k})_{pq}|^{s} + O(N^{-s/2})+O((d_{N}/(\eta N))^{s})
\end{equation}
where we used that $\mathbb{E}|W_{ij}|^{4s}$ is bounded. Note that the obtained error term is smaller than $(d_{N}/(\eta N))^{s/2}$ found in \eqref{gpq}. This completes the proof of the Lemma.
\end{proof}

As is suggested by the structure of the terms \eqref{s11}-\eqref{s16}, in what follows we will encounter many sums of the following generic form
\begin{equation}
\Omega := \sum_{\zeta_{1},\zeta_{2},\ldots,\zeta_{r}}\phi(W)\prod_{m \in I}\mathbb{E}_{k}(G^{(m)}_{k})_{\zeta_{m},\zeta_{m+1}} \label{omega}
\end{equation}
for some index set $I$ and a complex valued function $\phi$. By simply counting the number of occurences of $(G^{(m)}_{k})_{\zeta_{m},\zeta_{m+1}}$ we obtain
\begin{lemma}[Trivial bound]
\label{le:trivbound}
Suppose that $\mathbb{E}|\phi(W)|^{2}$ is uniformly bounded. Then
\begin{equation}
\label{trivbound}
\mathbb{E}|\Omega| \leq cN^{r}(d_{N}/N)^{|I|/2}
\end{equation}
with the same result holding if any of the factors in \eqref{omega} are replaced with the perturbation $G^{(m)}_{k\zeta_{a}\zeta_{b}}$ for any $a,b \in I$.
\end{lemma}

\begin{proof}
This follows by repeatedly applying the Cauchy-Schwarz inequality in conjunction with Lemmas \ref{le:offdiag} and \ref{le:gkij}.
\end{proof}

We start by giving some first estimates on the terms $\mathcal{S}_{1,1}$ to $\mathcal{S}_{1,6}$ defined in \eqref{s11}-\eqref{s16}.
\begin{lemma}[Estimates of $\mathcal{S}_{1,1}$, $\mathcal{S}_{1,3}$, $\mathcal{S}_{1,4}$ and $\mathcal{S}_{1,6}$]
We have the following bounds, holding uniformly in $k=1,\ldots,N$:
\begin{align}
\mathcal{S}_{1,1} &= -(N-k)s(z_{2})+O_{L^1}((d_{N}N^{-1})^{1/2}) \label{s11bound}\\
\mathcal{S}_{1,3} &= -s(z_{1})\mathcal{S}_{1}+O_{L^1}(d_{N}/N)\label{s13bound}\\
\mathbb{E}|\mathcal{S}_{1,4}| &\leq c\sqrt{d_{N}^{3}N^{-1}} \label{s14bound}\\
\mathbb{E}|\mathcal{S}_{1,6}| &\leq c\sqrt{d_{N}^{3}N^{-1}} \label{s16bound}
\end{align}
\end{lemma}

\begin{proof}
These estimates are straightforward consequences of the trivial bound of Lemma \ref{le:trivbound}.
\end{proof}
With $\mathcal{S}_{1,5}$ we have to take a bit more care because the trivial bound actually produces a bound of order $d_{N}$ which is divergent. To fix this we have to exploit independence and the fact that $\mathbb{E}|W_{pj}|^{2}=1$.
\begin{lemma}[Estimate for $\mathcal{S}_{1,5}$]
We have the following bound, holding uniformly in $k=1,\ldots,N$:
\begin{equation}
\mathbb{E}|\mathcal{S}_{1,5}| \leq c \sqrt{d_{N}^{2}/N}
\end{equation}
\end{lemma}
\begin{proof}
First denote
\begin{equation}
R_{p} := \sum_{q<k}(G^{(1)}_{k})_{pq}\mathbb{E}_{k}(G^{(2)}_{k})_{qp}
\end{equation}
and note that the usual estimates imply that $\mathbb{E}|R_{p}|^{2} = O(d_{N}^{2})$. Now by Cauchy-Schwarz we can estimate $\mathcal{S}_{1,5}$ as
\begin{equation}
\begin{split}
\mathbb{E}|\mathcal{S}_{1,5}| &\leq N^{-2}\sum_{p<k}\left(\mathbb{E}|R_{p}|^{2}\,\mathbb{E}\sum_{j1\neq k,j2\neq k}c_{pj1}c_{pj2}(|W_{pj1}|^{2}-1)(|W_{pj2}|^{2}-1)\right)^{1/2}\\
&=N^{-2}\sum_{p<k}\left(\mathbb{E}|R_{p}|^{2}\sum_{j \neq k}c_{pj}^{2}\mathbb{E}(|W_{pj}|^{2}-1)^{2}\right)^{1/2}\leq c\sqrt{d_{N}^{2}/N}
\end{split}
\end{equation}
\end{proof}
To complete the proof of Proposition \ref{prop:s1} the main task is to control $\mathcal{S}_{1,2}$. The problem is that the matrix $G^{(2)}_{k}$ still depends on $W_{pj}$, so in the following we will replace $G^{(2)}_{k}$ with $G^{(2)}_{kpj}$. Remarkably, the leading order contribution will come from the error in making this replacement.
\begin{lemma}
\label{le:s12}
We have the following bound, holding uniformly in $k=1,\ldots,N$:
\begin{equation}
\label{s12est}
\mathcal{S}_{1,2} = -(N-k)N^{-1}s(z_{2})\mathcal{S}_{1,k}(z_{1},z_{2})+O_{L^1}((d_{N}^{3}N^{-1})^{1/2})
\end{equation}
\end{lemma}

\begin{proof}
We start by replacing $(G^{(2)}_{k})_{qp}$ with $(G^{(2)}_{kpj})_{qp}$ in the definition \eqref{s12} of $\mathcal{S}_{1,2}$ and denote the modified sum by $\tilde{\mathcal{E}}_{1,2}$. We will show that $\tilde{\mathcal{E}}_{1,2}$ converges to $0$ in $L^{2}$. We have
\begin{equation}
\begin{split}
\label{s21error}
\mathbb{E}|\tilde{\mathcal{E}}_{1,2}|^{2} &= N^{-3}\mathbb{E}\sum_{\substack{j_1,j_2,p_1\\p_2,q_1,q_2}}W_{p_1j_1}\overline{W_{p_2j_2}}\mathbb{E}_{k}(G^{(1)}_{kp_1j_1})_{j_1q_1}\mathbb{E}_{k}(G^{(2)}_{kp_1j_1})_{q_1p_1}\\
&\times\overline{\mathbb{E}_{k}(G^{(1)}_{kp_2j_2})_{j_2q_2}}\overline{\mathbb{E}_{k}(G^{(2)}_{kp_2j_2})_{q_2p_2}}
\end{split}
\end{equation}
where here and unless otherwise stated, all indices in the summation run from $1$ to $k-1$. To make $W_{p_1j_1}$ and $G_{kp_2j_2}$ independent we use a similar perturbation formula to remove the matrix element $W_{p_1j_1}$ from $G_{kp_2j_2}$. Therefore, we define $G^{(2),p_1j_1}_{kp_2j_2}$ as the resolvent of the matrix $W$ with the $k^{\mathrm{th}}$ column and row erased and with entries $W_{p_1j_1}, W_{j_1p_1}, W_{p_2j_2}$ and $W_{j_2p_2}$ replaced with $0$. It is easy to show that we have a similar identity
\begin{align}
\label{2ndorderpert}
(G^{(2)}_{kp_2j_2})_{ab} &= (G^{(2),p_1j_1}_{kp_2j_2})_{ab}-c_{p_1j_1}W_{p_1j_1}N^{-1/2}(G^{(2),p_1j_1}_{kp_2j_2})_{ap_1}(G^{(2)}_{kp_2j_2})_{j_1b}\\
&-c_{p_1j_1}W_{j_1p_1}N^{-1/2}(G^{(2),p_1j_1}_{kp_2j_2})_{aj_1}(G^{(2)}_{kp_2j_2})_{p_1b}
\end{align}
Then if we replace the last two factors in \eqref{s21error} with \eqref{2ndorderpert} the main term has identically zero expectation unless both $p_1=j_1$ and $p_2=j_2$, which we assume for the moment does not hold. The higher order terms in \eqref{2ndorderpert} give rise to an error
\begin{align}
A &= N^{-7/2}\mathbb{E}\sum_{\substack{j_1,j_2,p_1\\p_2,q_1,q_2}}W_{p_1j_1}\overline{W_{p_2j_2}}\mathbb{E}_{k}(G^{(1)}_{kp_1j_1})_{j_1q_1}\mathbb{E}_{k}(G^{(2)}_{kp_1j_1})_{q_1p_1}\\
& \times\overline{\mathbb{E}_{k}(G^{(1),p_1j_1}_{kp_2j_2})_{j_2q_2}}\overline{\mathbb{E}_{k}(W_{p_1 j_1}(G^{(2),p_1j_1}_{kp_2j_2})_{q_2j_1}(G^{(2)}_{kp_2j_2})_{p_1p_2})}
\end{align}
and three further error terms which do not differ in any important way from $A$ above. To estimate $A$, we now force independence in $W_{p_2,j_2}$ by replacing $G^{(2)}_{kp_1j_1}$ with $G^{(2),p_2j_2}_{kp_1j_1}$. Again the main term obtained by this replacement has expectation $0$ because $\mathbb{E}\overline{W_{p_2j_2}}=0$. The error terms in \eqref{2ndorderpert} give rise to sums of the form
\begin{align}
A' &= N^{-4}\mathbb{E}\sum_{\substack{j_1,j_2,p_1\\p_2,q_1,q_2}}W_{p_1j_1}\overline{W_{p_2j_2}}\mathbb{E}_{k}(G^{(1),p_2j_2}_{kp_1j_1})_{j_1q_1}\mathbb{E}_{k}(W_{p_2 j_2}(G^{(2),p_2j_2}_{kp_1j_1})_{q_1j_2}(G^{(2)}_{kp_1j_1})_{p_2p_1})\\
&\times\overline{\mathbb{E}_{k}(G^{(1),p_1j_1}_{kp_2j_2})_{j_2q_2}}\overline{\mathbb{E}_{k}(W_{p_1 j_1}(G^{(2),p_1j_1}_{kp_2j_2})_{q_2j_1}(G^{(2)}_{kp_2j_2})_{p_1p_2})}
\end{align}
We call such a term \textit{maximally expanded} because we can no longer exploit independence of the different factors to reduce the size of the sum (none of the factors have zero expectation at this point). See also \cite{EK13max} for related methods. By the prescription \eqref{trivbound}, we have
\begin{equation}
A' \leq cN^{-4}N^{6}(d_{N} N^{-1})^{3} \leq c d_{N}^{3}N^{-1}
\end{equation}
It is clear that all error terms resulting from the replacement \eqref{2ndorderpert} give the same bounds after employing this procedure. 

Now consider the diagonal terms $p_{1}=p_{2}$ and $j_{1}=j_{2}$ contributing to $\tilde{\mathcal{E}}_{1,2}$:
\begin{align}
&N^{-3}\mathbb{E}\sum_{j,p,q_1,q_2}|W_{pj}|^{2}\mathbb{E}_{k}(G^{(1)}_{kpj})_{jq_1}\mathbb{E}_{k}(G^{(1)}_{kpj})_{q_1p}\overline{\mathbb{E}_{k}(G^{(2)}_{kpj})_{jq_2}}\overline{\mathbb{E}_{k}(G^{(2)}_{kpj_2})_{q_2p}}\\
&\leq cN^{-3}N^{4}(d_{N} N^{-1})^{2} \leq cd_{N}^{2}/N
\end{align}
We conclude that $\mathbb{E}|\tilde{\mathcal{E}}_{1,2}|^{2} \leq cd_{N}^{3}/N$. Now again using \eqref{gkijpert} we have
\begin{align}
\mathcal{S}_{1,2} &= N^{-3/2}\sum_{\substack{j < k\\p<k,q<k}}W_{pj}\mathbb{E}_{k}(G^{(1)}_{kpj})_{jq}\mathbb{E}_{k}((G^{(2)}_{k})_{qp}-(G^{(2)}_{kpj})_{qp}) + O_{p}(d_{N}^{3}N^{-1})\\
&=-N^{-2}\sum_{\substack{j < k\\p<k,q<k}}c_{pj}W_{pj}^{2}\mathbb{E}_{k}(G^{(1)}_{kpj})_{jq}\mathbb{E}_{k}((G^{(2)}_{kpj})_{qp}(G^{(2)}_{k})_{jp}) \label{negs12}\\
&-N^{-2}\sum_{\substack{j < k\\p<k,q<k}}c_{pj}|W_{pj}|^{2}\mathbb{E}_{k}(G^{(1)}_{kpj})_{jq}\mathbb{E}_{k}((G^{(2)}_{kpj})_{qj}(G^{(2)}_{k})_{pp})+O_{p}(d_{N}^{3}N^{-1}) \label{finals12}
\end{align}
By \eqref{trivbound}, we see that \eqref{negs12} is $O_{L^1}(\sqrt{d_{N}^{3}N^{-1}})$. The final term \eqref{finals12} (let us denote it $\tilde{\mathcal{S}}_{1,2}$) can be written
\begin{align}
\tilde{\mathcal{S}}_{1,2} =& -N^{-2}s(z_{2})\sum_{j,p,q<k}c_{pj}\mathbb{E}_{k}(G^{(1)}_{kpj})_{jq}\mathbb{E}_{k}(G^{(2)}_{kpj})_{qj} \label{s12dom}\\
&-N^{-2}s(z_{2})\sum_{j,p<k}(|W_{pj}|^{2}-1)c_{pj}\sum_{q<k}\mathbb{E}_{k}(G^{(1)}_{kpj})_{jq}\mathbb{E}_{k}(G^{(2)}_{k})_{qj} \label{s12midneg}\\
&-N^{-2}\sum_{j,p,q<k}c_{pj}|W_{pj}|^{2}\mathbb{E}_{k}(G^{(1)}_{kpj})_{jq}\mathbb{E}_{k}(G^{(2)})_{qj}((G^{(2)}_{k})_{pp}-s(z_{2})) \label{s12finneg}
\end{align}
To estimate the term \eqref{s12midneg}, first replace $G^{(2)}_{k}$ with $G^{(2)}_{kpj}$ and note that by \eqref{trivbound} the error in making this replacement is $O_{p}(d_{N}N^{-1/2})$. We denote the remaining sum $\tilde{\mathcal{F}}_{1,2}$. We have
\begin{align}
\mathbb{E}|\tilde{\mathcal{F}}_{1,2}|^{2} &= N^{-4}|s(z_{2})|^{2}\sum_{\substack{j_1,p_1,q_1\\j_2,p_2,q_2}}(|W_{p_1j_1}|^{2}-1)(|W_{p_2j_2}|^{2}-1)\mathbb{E}_{k}(G^{(1)}_{kp_1j_1})_{j_1q_1}\mathbb{E}_{k}(G^{(2)}_{kp_1j_1})_{q_1j_1}\\
&\times \overline{\mathbb{E}_{k}(G^{(1)}_{kp_2j_2})_{j_2q_2}\mathbb{E}_{k}(G^{(2)}_{kp_2j_2})_{q_2j_2}}
\end{align}
This sum has a similar structure to that seen in the computation of $\mathbb{E}|\tilde{\mathcal{E}}_{1,2}|^{2}$. Applying exactly the same procedure shows that $\mathbb{E}|\tilde{\mathcal{F}}_{1,2}|^{2} \leq cd_{N}^{3}/N^{2}$. The term \eqref{s12finneg} is $O_{L^1}(\sqrt{d_{N}^{3}N^{-1}})$ as follows directly from the generic bound \eqref{trivbound}. In \eqref{s12dom} all terms in the sum where $p=j$ will only contribute $O_{L^1}(d_{N} N^{-1})$ so can be neglected, which justifies us setting $c_{pj}=1$ in \eqref{s12dom}. 

Finally, note that \eqref{s12dom} has a similar form to $\mathcal{S}_{1,k}$ except with the perturbed resolvent elements $(G^{(1)}_{kpj})_{jq}$. By the trivial bound \eqref{trivbound} they can be exchanged with the original ones $(G^{(1)}_{k})_{jq}$ at a cost $O_{L^1}((d_{N}N^{-1}))$. Then the summation over $p$ just gives a prefactor $(N-k)$, leading to \eqref{s12est}.
\end{proof}

\subsection{Proof of Proposition \ref{prop:findim}}
\label{se:conv}
With the most significant challenges dealt with in the previous subsection, our aim now is to complete the proof of Proposition \ref{prop:findim} by solving relation \eqref{recs1} and computing the limit $N \to \infty$. In other words, we finally prove the conditional variance formula \eqref{covarconv} which is enough to verify Proposition \ref{prop:findim}.

\begin{proposition}
\label{prop:cn}
Consider the approximate covariance kernel $\tilde{C}_{N}(z_{1},z_{2})$ given by \eqref{apcoker}. Then, in terms of the notation \eqref{zn}, we have the estimate
\begin{equation}
\tilde{C}_{N}(z_{1},\overline{z_{2}}) = \frac{1}{(i(\tau_1-\tau_2+i(\eta_1+\eta_2)))^{2}} + O_{L^1}(1/d_{N})+O_{L^1}\left(\log(d_{N})\sqrt{d_{N}^{3}/N}\right)
\label{cnlimitprop}
\end{equation}
and
\begin{equation}
\tilde{C}_{N}(z_{1},z_{2}) = O_{L^1}(1/d_{N})+O_{L^1}\left(\log(d_{N})\sqrt{d_{N}^{3}/N}\right) \label{zerolim}
\end{equation}
\end{proposition}
\begin{proof}
Recall that $\tilde{C}_{N}(z_{1},z_{2})$ can be expressed in terms of the auxiliary kernel $K_{N}(z_{1},\overline{z_{2}})$ in \eqref{s1k}:
\begin{equation}
\tilde{C}_{N}(z_{1},z_{2}) = \frac{1}{d_{N}^{2}}\frac{\partial^{2}}{\partial z_{1} \partial z_{2}}s(z_{1})s(z_{2})K_{N}(z_{1},z_{2})
\end{equation}
In turn, $K_{N}(z_{1},z_{2})$ was expressed in terms of the sum $\mathcal{S}_{1,k}(z_{1},z_{2})$ defined in \eqref{s1} and computed in Proposition \ref{prop:s1}. Let us denote the error term in that Proposition by $E_{N,k}$. Then we have
\begin{equation}
\mathcal{S}_{1,k}(z_{1},z_{2}) = \frac{N-k}{N}\frac{s(z_{1})s(z_{2})}{1-\frac{N-k}{N}s(z_{1})s(z_{2})}+\frac{E_{N,k}}{1-\frac{N-k}{N}s(z_{1})s(z_{2})} 
\end{equation}
where $E_{N,k} = O_{L^1}(\sqrt{d_{N}^{3}/N})$ uniformly in $k$. Then by definition \eqref{s1} the kernel $K_{N}(z_{1},z_{2})$ can be written as
\begin{equation}
K_{N}(z_{1},z_{2}) = \frac{1}{N}\sum_{k=1}^{N}\frac{N-k}{N}\frac{s(z_{1})s(z_{2})}{1-\frac{N-k}{N}s(z_{1})s(z_{2})}+\frac{1}{N}\sum_{k=1}^{N}\frac{E_{N,k}}{1-\frac{N-k}{N}s(z_{1})s(z_{2})} + O_{L^1}(d_{N}/N) \label{cnasympt}
\end{equation}
These sums are close to Riemann integrals, but before we calculate them we must take into account a subtle feature of the mesoscopic regime: when the points $z_{1}$ and $z_{2}$ have opposing signs in their imaginary parts, there is a singularity in the denominators of \eqref{cnasympt}, due to the asymptotic formula
\begin{equation}
1-s(z_{1})s(\overline{z_{2}}) = \frac{1}{d_{N}}\frac{\eta_1+\eta_2+(\tau_2-\tau_1)i}{\sqrt{4-E^{2}}}+O(d_{N}^{-2})\label{singasympt}
\end{equation}
If the signs of the imaginary part are the same, there is no singularity and one finds that the limiting covariance is identically zero, as in \eqref{zerolim}. Now let us control the errors in \eqref{cnasympt}, assuming there is a singularity. For $0 < u < 1$, let $\psi(u) := |1-us(z_{1})s(\overline{z_{2}})|^{-1}$. Then standard results about the Riemann integral show that the error term in \eqref{cnasympt} is bounded in $L^1$ by
\begin{equation}
\sqrt{d_{N}^{3}/N}\frac{1}{N}\sum_{k=1}^{N}\frac{1}{|1-\frac{N-k}{N}s(z_{1})s(\overline{z_{2}})|}\leq \sqrt{d_{N}^{3}/N}\left(\int_{0}^{1}\psi(u)\,du + |E_{N,k}|\frac{\norm{\psi}_{\mathrm{TV}}}{N}\right)
\end{equation}
where $\norm{\psi}_{\mathrm{TV}}$ is the total variational norm of the function $\psi$. Since $us(z_{1})s(z_{2})$ has non-zero imaginary part, $\psi'(u)$ exists and is Riemann integrable. Then the total variational norm can be written
\begin{equation}
\norm{\psi}_{\mathrm{TV}} = \int_{0}^{1}|\psi'(u)|\,du
\end{equation}
and a simple calculation taking into account the asymptotics \eqref{singasympt} shows that $\norm{\psi}_{\mathrm{TV}} = O(d_{N})$ as $N \to \infty$, while $\int_{0}^{1}\psi(u)\,du = O(\log(d_{N}))$ as $N \to \infty$.

The error from approximating the first term in \eqref{cnasympt} can be estimated in the same way, and we get 
\begin{equation}
\label{intlog}
\begin{split}
K_{N}(z_{1},\overline{z_{2}}) &= \int_{0}^{1}\frac{us(z_{1})s(\overline{z_{2}})}{1-us(z_{1})s(\overline{z_{2}})}\,du+O_{L^1}\left(\log(d_{N})\sqrt{d_{N}^{3}/N}\right)\\
&= -1-\frac{\log(1-s(z_{1})s(\overline{z_{2}}))}{s(z_{1})s(\overline{z_{2}})}+O_{L^1}\left(\log(d_{N})\sqrt{d_{N}^{3}/N}\right)
\end{split}
\end{equation}
Inserting \eqref{intlog} into the definition of $\tilde{C}_{N}(z_{1},\overline{z_{2}})$ and bearing in mind Remark \ref{rem:cauchy}, we have
\begin{align}
\tilde{C}_{N}(z_{1},\overline{z_{2}}) &:= \frac{1}{d_{N}^{2}}\frac{\partial^{2}}{\partial z_{1} \partial z_{2}}s(z_{1})s(\overline{z_{2}})K_{N}(z_{1},\overline{z_{2}})\\
&= \frac{1}{d_{N}^{2}}\frac{s'(z_{1})s'(\overline{z_{2}})}{(1-s(z_{1})s(\overline{z_{2}}))^{2}}-\frac{1}{d_{N}^{2}}s'(z_{1})s'(z_{2})+O_{L^1}\left(\log(d_{N})\sqrt{d_{N}^{3}/N}\right)\\
&= \frac{1}{(i(\tau_1-\tau_2+i(\eta_1+\eta_2)))^{2}}+O(1/d_{N})+O_{L^1}\left(\log(d_{N})\sqrt{d_{N}^{3}/N}\right)
\end{align}
where in the last line we applied formula \eqref{singasympt} and used that
\begin{equation}
s'(z_{1})s'(\overline{z_{2}}) = \frac{1}{4-E^{2}}+O(d_{N}^{-1}).
\end{equation}
This completes the proof of Proposition \ref{prop:cn}. Therefore we have finally verified the conditions \eqref{covarconv} and \eqref{covarconv2}, hence completing the proof of Proposition \ref{prop:findim}.
\end{proof}

\section{Tightness and linear statistics}
\label{se:testfn}
The goal of this section is to extend the CLT obtained in the previous section to a wider range of test functions, focusing on the distributional convergence of \eqref{lsmeso}. Our approach will be to combine the Helffer-Sj\"ostrand formula \eqref{genint} with the main result of the last Section, telling us that the finite-dimensional distributions of $V_{N}$ converge to $\Gamma'^{+}_{0}$. If we can interchange the integrals in \eqref{genint} with this distributional convergence then the CLT for mesoscopic linear statistics \eqref{lsmeso} would be proved. 

In practise we will need to prove a certain tightness criteria in order to justify this interchange, involving uniform estimates on the second moment $\mathbb{E}|V_{N}(\tau+i\eta)|^{2}$. We will handle those estimates mainly with Theorem \ref{thm:optimalbound}, but this Theorem only applies in the bulk region with $\eta \geq d_{N}/N$. For smaller $\eta$ we need the following variance bound, due to Sosoe and Wong (in our notation).
\begin{proposition}
\label{prop:sw}
	Let $\epsilon > 0$. Then for $0 < \eta < 1$, and $|E + \tau/d_N|\leq 5$ we have that there is a universal constant $C > 0$ such that:
\begin{equation}
	\mathbb{E}|V_N(\tau+i\eta)|^2 \leq C d_N^\epsilon \eta^{-2-\epsilon}, \label{swbound}
\end{equation}
\end{proposition}
\begin{proof}
See Proposition 4.1 in \cite{SW13}.
\end{proof}

The reader should compare this result with that obtained from Theorem \ref{thm:optimalbound} which gives an improved bound $\mathbb{E}|V_N(t+i\eta)|^2 \leq C\eta^{-2}$ (Proposition \ref{prop:bulkBound}) but in a more restricted region.

To prove the CLT for \eqref{lsmeso}, we consider two situations. Firstly we consider the case that $f \in C^{1,\alpha}_{c}(\mathbb{R})$, the H\"older space of compactly supported functions $f$ such that $f'$ is H\"older continuous with exponent $\alpha>0$. The hypothesis of compact support will allow us to avoid complications coming from the edges of the spectrum and will serve as a warm up, illustrating our general approach. 

In the last subsection we consider the more challenging case where the hypothesis of compact support is replaced with a more general decay condition $f(x)$ and $f'(x)$. In particular we will prove our main result, Theorem \ref{th:maintheorem}. 

\subsection{Compactly supported functions with H\"older continuous first derivative}
To obtain the CLT we will apply the Helffer-Sj\"ostrand formula \eqref{genint} with
\begin{equation}
\Psi_{f}(t,\eta) = (f(t)+i(f(t+\eta)-f(t)))J(\eta)
\end{equation}
where $J(\eta)$ is a smooth function of compact support, equal to $1$ in a neighbourhood of $\eta=0$ and equal to $0$ if $\eta>1$, see Lemma \ref{lem:extension}. It follows that
\begin{equation}
\overline{\partial} \Psi_{f}(t,\eta) = (f'(t)-f'(t+\eta)+i(f'(t+\eta)-f'(t)))J(\eta)+(if(t)-(f(t+\eta)-f(t)))J'(\eta)
\end{equation}
and if we assume that $f'$ is H\"older continuous with exponent $0 < \alpha < 1$ then we also have the bound $|\overline{\partial} \Psi_{f}(t,\eta)| \leq c\eta^{\alpha}$ for some constant $c>0$ for $\eta$ small. Also $\Psi_{f}(t,\eta)$ satisfies $\Psi_{f}(t,0)=f(t)$, and that $\overline{\partial}\Psi_{f}(t,\eta)$ is compactly supported whenever $f$ is.


\begin{theorem}
\label{thm:c2c}
	Suppose that $d_{N}$ satisfies the condition $1 \ll d_{N} \ll N^{1/3}$ and consider compactly supported test functions $f_{1},\ldots,f_{M}$ whose first derivatives are H\"older continuous for some exponent $\alpha>0$. Then for any fixed $E \in (-2,2)$ in \eqref{lsmeso} we have the convergence in distribution 
\begin{equation}
	(\tilde{X}^{\mathrm{meso}}_{N}(f_{1}),\ldots, \tilde{X}^{\mathrm{meso}}_{N}(f_{M})) \to (X_{1},\ldots,X_{M})
\end{equation}
where $(X_{1},\ldots,X_{M})$ is an $M$-dimensional Gaussian vector with covariance matrix
\begin{equation}
 \mathbb{E}(X_{p}X_{q}) = \frac{1}{2\pi}\int_{-\infty}^{\infty}|k|\,\hat{f_{p}}(k)\overline{\hat{f_{q}}(k)}\,dk, \qquad 1 \leq p,q \leq M \label{covarc2c}
\end{equation}
\end{theorem}

\begin{proof}
Since $f$ has compact support, we may write 
\begin{equation}
\tilde{X}^{\mathrm{meso}}_{N}(f) = \frac{1}{\pi}\re\int_{0}^{1}\int_{\tau_{-}}^{\tau_{+}}V_{N}(\tau+i\eta)\overline{\partial} \Psi_{f}(\tau,\eta)\,d\tau\,d\eta 
\end{equation}
for some fixed $\tau_{-}$ and $\tau_{+}$. First note that the region $0 \leq \eta \leq d_{N}/N$ can be neglected, since
\begin{align}
&\int_{0}^{d_{N}/N}\int_{\tau_{-}}^{\tau_{+}}\mathbb{E}(|V_{N}(\tau+i\eta)|)|\overline{\partial} \Psi_{f}(\tau,\eta)|\,d\tau,d\eta \leq cd_{N}^{\epsilon}\int_{0}^{d_{N}/N}\int_{\tau_{-}}^{\tau_{+}}\eta^{\alpha-1-\epsilon}\,d\tau\,d\eta \label{swb1}\\
&= c(\tau_{+}-\tau_{-})N^{\epsilon}(d_{N}/N)^{\alpha}
\end{align}
The latter goes to zero after choosing $0 < \epsilon < (1-\gamma)\alpha$, where $d_{N} = N^{\gamma}$ with $0 < \gamma < 1$. Therefore, it suffices to study the convergence in distribution of
\begin{equation}
 X(V_{N},f) = \frac{1}{\pi}\mathrm{Re}\int_{0}^{1}\int_{\tau_{-}}^{\tau_{+}}V_{N}(\tau+i\eta)\chi_{N}(\eta)\overline{\partial} \Psi_{f}(\tau,\eta)\,d\tau\,d\eta
\end{equation}
where $\chi_{N}(\eta)$ is an indicator function, equal to $1$ on the region $d_{N}/N \leq \eta \leq 1$ and $0$ otherwise. We wish to show that $X(V_{N},f)$ converges in distribution to $X(\Gamma'^{+}_{0},f)$. Since we proved in the previous section that the finite-dimensional distributions of $V_{N}(\tau+i\eta)$ converge pointwise to $\Gamma'^{+}_{0}(\tau+i\eta)$, we can appeal to Theorem \ref{thm:weakconvInt}. Let $\Phi$ denote the set of functions $\phi(z,w): \mathbb{H}\times \mathbb{C} \to \mathbb{C}$ of the form $\phi(z,w) = \overline{\partial}\Psi_g(z) w$ where $g$ is a function in the class stated by Theorem \ref{thm:c2c}  and $M$ is as in the Theorem \ref{thm:weakconvInt}.

Let $D$ be the domain $[0,1] \times [\tau_{-},\tau_{+}]$. Then Theorem \ref{thm:weakconvInt} guarantees the convergence in distribution $(X(V_{N},f_1),\ldots,X(V_{N},f_M)) \to (X(\Gamma'^{+}_{0},f_1),\ldots,X(\Gamma'^{+}_{0},f_M))$ provided we check the following tightness conditions:
\begin{align}
\inf_{B\subset \mathbb{H}, \lambda(B)<\infty} \limsup_{N\to\infty}\mathbb{P}\left(\frac{1}{\pi}\mathrm{Re}\int_{D \setminus B}|V_{N}(\tau+i\eta)\chi_{N}(\eta)\overline{\partial}\Psi_{f}(\tau,\eta)|\,d\tau\,d\eta \geq \epsilon  \right) &= 0 \label{limBcc} \\
\lim_{K \to \infty} \limsup_{N \to \infty}\mathbb{P}\left(\frac{1}{\pi}\mathrm{Re}\int_{D}(|V_{N}(\tau+i\eta)\chi_{N}(\eta)\overline{\partial}\Psi_{f}(\tau,\eta)|-K)^+\,d\tau\,d\eta \geq \epsilon \right) &= 0, \label{limKcc}
\end{align}
First we prove \eqref{limBcc}. By Markov's inequality it suffices to check that
\begin{equation}
 \inf_{B\subset D, \lambda(B)<\infty} \limsup_{N\to\infty}\int_{D \setminus B}\mathbb{E}(|V_{N}(\tau+i\eta)|)\chi_{N}(\eta)|\overline{\partial}\Psi_{f}(\tau,\eta)|\,d\tau\,d\eta = 0 \label{limBccmark}
\end{equation}
Now since $\tau$ is fixed, the real part appearing in the denominator of the resolvent is bounded away from the edges (due to the compact support of $f$). Also the imaginary part appearing in the denominator is no less than $1/N$. Hence we can apply Proposition \ref{prop:bulkBound} and obtain $\mathbb{E}|V_{N}(\tau+i\eta)| \leq c\eta^{-1}$. By the assumptions on $f$, $\eta^{-1}|\overline{\partial}\Psi_{f}(\tau,\eta)|$ is integrable on $D \setminus B$, and so \eqref{limBccmark} is bounded by
\begin{equation}
	\inf_{B\subset D, \lambda(B)<\infty} \int_{D \setminus B}c\eta^{-1}|\overline{\partial}\Psi_f(\tau,\eta)|\, d\tau\,d\eta = 0
\end{equation}
where the last equality follows from the dominated convergence theorem. To check $\eqref{limKcc}$ we proceed similarly, noting that for any $\delta>0$, we have
\begin{equation}
 (|V_{N}(\tau+i\eta)\chi_{N}(\eta)\overline{\partial}\Psi_{f}(\tau,\eta)|-K)^+ \leq \frac{1}{K^{\delta}}\,|V_{N}(\tau+i\eta)\chi_{N}(\eta)\overline{\partial}\Psi_{f}(\tau,\eta)|^{1+\delta}
\end{equation}
Then \eqref{limKcc} is bounded by
\begin{equation}
\begin{split}
&\lim_{K \to \infty}K^{-\delta} \limsup_{N \to \infty}\int_{D}\mathbb{E}(|V_{N}(\tau+i\eta)|^{1+\delta})\chi_{N}(\eta)|\overline{\partial}\Psi_{f}(\tau,\eta)|^{1+\delta}\,d\tau\,d\eta\\
&\leq \lim_{K \to \infty}K^{-\delta}\int_{D}c_1\eta^{-1-\delta}|\overline{\partial}\Psi_f(\tau,\eta)|^{1+\delta}\,d\tau\,d\eta = 0 \label{limKcc0}
\end{split}
\end{equation}
where in \eqref{limKcc0}, we choose $\delta$ so that $0 < \delta < \alpha/(1-\alpha)$ so that the above integral is finite by the behavior $|\overline{\partial}\Psi_f(\tau,\eta)| \leq c\eta^{-\alpha}$ for small $\eta$. This completes the proof that $X(V_{N},f)$ converges in distribution to the random variable
\begin{equation}
X(\Gamma'^{+}_{0},f) = \frac{1}{\pi}\mathrm{Re}\int_{D}\Gamma'^{+}_{0}(\tau+i\eta)\overline{\partial} \Psi_{f}(\tau+i\eta)\,d\tau\,d\eta,
\end{equation}
Since the integral of a Gaussian process is Gaussian, it just remains to compute the covariance and verify it is well-defined, see the next Lemma.
\end{proof}

\begin{lemma}
	Consider the random functional defined by: 
	\begin{equation}	
		X(\Gamma'^{+}_{0},f) = \frac{1}{\pi}\mathrm{Re}\int_{\mathbb{H}}\Gamma'^{+}_{0}(\tau+i\eta)\overline{\partial} \Psi_{f}(\tau+i\eta)\,d\tau\,d\eta \label{xfrep}.
	\end{equation}
	Then the covariance $\mathbb{E}X(\Gamma'^{+}_{0},f_1)X(\Gamma'^{+}_{0},f_2)$ is given by \eqref{covar}, provided that $f_1$ and $f_2$ are in $C^{1,\alpha}(\mathbb{R})$ for some $\alpha > 0$, and $|f_i|$, $|f_i'|$ are $O(|x|^{-(1+\beta)})$ for some $\beta > 0$.
\end{lemma}
\begin{proof}
Recall that the process $\Gamma^{+}_{0}(\tau+i\eta)$ for $\eta > 0$ appearing in $\eqref{xfrep}$ has the covariance structure
\begin{equation}
 \mathbb{E}(\Gamma'^{+}_{0}(\tau_1+i\eta_1)\overline{\Gamma'^{+}_{0}(\tau_2+i\eta_2)}) = \frac{1}{(i(\tau_{1}-\tau_{2}+i(\eta_{1}+\eta_{2})))^{2}} \label{covzeta}
\end{equation}
while $\mathbb{E}(\Gamma'^{+}_{0}(\tau_1+i\eta_1)\Gamma'^{+}_{0}(\tau_2+i\eta_2)) = 0$. We compute
\begin{align}
	\mathbb{E}&(X(\Gamma_0'^+,f_1),X(\Gamma_0'^+,f_2))\\
	=& \frac{1}{2\pi^2}\mathbb{E}\bigg[\int_{\mathbb{H}\times \mathbb{H}} \,dz_1 \,dz_2 \bigg\{\re\left[\Gamma'^{+}_0(z_1)\Gamma'^{+}_0(z_2)\overline{\partial}\Psi_{f_1}(z_1) \overline{\partial}\Psi_{f_2}(z_2)\right] \label{covcontr1} \\
		& + \re\left[\Gamma'^{+}_0(z_1)\overline{\partial}\Psi_{f_1}(z_1) \overline{\Gamma'^{+}_0(z_2)\overline{\partial}\Psi_{f_2}(z_2)}\right]\bigg\} \bigg], \label{covcontr2}
\end{align}
we proceed to check that the expectation is finite and that we can exchange expectation with integrals by utilizing Fubini's theorem. It suffices to verify that
\begin{equation}
	\int_{\mathbb{H}\times \mathbb{H}} \,dz_1\,dz_2 \mathbb{E}|\Gamma_0'^{+}(z_1)\Gamma_0'^{+}(z_2)| |\overline{\partial}\Psi_{f_1}(z_1)||\overline{\partial}\Psi_{f_2}(z_2)| < \infty,
\end{equation}
by using Cauchy-Schwarz, the above integral is bounded by
\begin{equation}
	\int_{\mathbb{H}\times \mathbb{H}} \,dz_1\,dz_2 \eta_1^{-1}\eta_2^{-1}|\overline{\partial}\Psi_{f_1}(z_1)||\overline{\partial}\Psi_{f_2}(z_2)| < \infty,
\end{equation}
since by our assumptions on $f_i$, we have
\begin{equation}
	|f_i'(t+\eta)-f_i'(t)|\leq \min(C_1\eta^\alpha, C_2 |t|^{-(1+\beta)}) \leq c \eta^{\alpha\sigma}|t|^{-(1+\beta)(1-\sigma)},
\end{equation}
for any $\sigma \in (0,1)$ and $c>0$ a constant independent of $t$ and $\eta$, giving us the bound
\begin{equation}
	|\overline{\partial}\Psi_{f_i}(t,\eta)| \leq c\eta^{\alpha\sigma}|t|^{-(1+\beta)(1-\sigma)},
\end{equation}
taking $\sigma < \beta/(1+\beta)$ gives a bound that is integrable (recall that $\overline{\partial}\Psi_{f_i}(t,\eta)$ is supported on $\eta \in [0,1]$). Hence, the covariance of $X(\Gamma_0'^{+},f)$ is finite, which implies, in particular that $X(\Gamma_0'^{+},f)$ is finite almost-surely --- yielding that the process is well-defined as required. 

After taking the expectation inside the integral, the first term \eqref{covcontr1} vanishes identically. In the second term \eqref{covcontr2}, we write the covariance in the Sobolev form
\begin{equation}
	\label{xflim} \mathbb{E}\Gamma'^{+}_{0}(\tau+i\eta)\overline{\Gamma'^{+}_{0}(\tau+i\eta)} = \frac{1}{2\pi}\int_{\mathbb{R}}\,dk |k|\widehat{r_{\tau_1,\eta_1}}(k)\overline{\widehat{r_{\tau_2,\eta_2}}(k)}
\end{equation}
where $r_{\tau_1,\eta_1}(x) = (x-\tau_1-i\eta_1)^{-1}$. We now interchange the integration over $k$ with the integration over $\mathbb{H} \times \mathbb{H}$. To justify this, note that $\widehat{r_{\tau_1,\eta_1}}(k) = -2\pi i e^{-ik\tau_1}e^{-|k|\eta_1}$ and due to the conditions on $f_1,f_2$, we have the bound:
\begin{equation}
	\int_{\mathbb{H}\times \mathbb{H}} \,dz_1 \,dz_2\int_\mathbb{R} \,dk |\overline{\partial}{\Psi}_{f_1}(z_1)| |\overline{\partial}\Psi_{f_2}(z_2)|		 |k|e^{-|k|\eta_1}e^{-|k|\eta_2} \leq C \int_0^1\int_0^1 \,d\eta_1\,d\eta_2 \frac{(\eta_1\eta_2)^{\alpha\sigma}}{(\eta_1+\eta_2)^{2}} < \infty.
\end{equation}
After interchanging these integrals, the integration over $\mathbb{H} \times \mathbb{H}$ factorises as a product. Now we would like to interchange the Fourier transform with the integral over $\mathbb{H}$, for which it suffices to bound the following
\begin{equation}
\label{convbound}
\int_{0}^{1}\,d\eta_1\int_{-\infty}^{\infty}\,dx \int_{-\infty}^{\infty}\,d\tau_1\,\bigg{|}\frac{\overline{\partial}\Psi_{f}(\tau_1,\eta_1)}{x-\tau_1-i\eta_1}\bigg{|}=\int_{0}^{1}\,d\eta_1 \norm{g_{\eta_1}*h_{\eta_1}}_{1}
\end{equation}
where $g_{\eta_1}(\tau_1) = |\tau_1+i\eta_1|^{-1}$ and $h_{\eta_1}(\tau_1) = |\overline{\partial}\Psi_{f}(\tau_1,\eta_1)|$. To bound the $L^{1}$ norm of the convolution, we apply Young's inequality $\norm{g_{\eta_1}*h_{\eta_1}}_{1} \leq \norm{g_{\eta_1}}_{p}\norm{h_{\eta_1}}_{q}$ with $q=1-\delta$ and $p=1+\delta/(1-\delta)$ with $\delta>0$. A simple computation shows that $\norm{g_{\eta_1}}_{p} \leq c\eta_1^{-\delta/(1-\delta)}$ while for sufficiently small $\delta$, $\norm{h_{\eta_1}}_{q}$ is bounded uniformly in $\eta_1$ due to the integrability assumptions on $f$ and its derivatives. This shows that $\norm{g_{\eta_1}*h_{\eta_1}}_{1} \leq c\eta_{1}^{-\delta/(1-\delta)}$ so that \eqref{convbound} is finite.

After performing all such interchanges of integration, we finally obtain
\begin{align}
	&\mathbb{E}(X(\Gamma_0'^{+},f_1)X(\Gamma_0'^+,f_2))\\
	&=\frac{1}{8\pi}\int_{-\infty}^{\infty}|k|\,dk\int_{-\infty}^{\infty}\,dx \,e^{-ikx}\frac{1}{\pi}\int_\mathbb{H} \,dz_1 \frac{1}{x-z_1}\overline{\partial}\Psi_{f_1}(z_1)\\
	&\times \overline{\int_{-\infty}^{\infty}\,dx\,e^{-ikx}\,\frac{1}{\pi}\int_\mathbb{H} \,dz_2\frac{1}{x-z_2}\overline{\partial}\Psi_{f_2}(z_2)}+\mathrm{c.c.} \label{covcomp}
\end{align}

Now the inner integrals over $\mathbb{H}$ can be evaluated by Lemma \ref{lem:extension}. There is a caveat however, Lemma \ref{lem:extension} requires the function $f$ to be compactly supported. We remedy this by taking our function $f$ and multiplying it by a cutoff function $\phi_n = \phi(x/n)$ where $\phi(x)$ is 1 on $[-1,1]$ and vanishes outside $[-2,2]$, we let $f_n = \phi_n f$. By Lemma \ref{lem:extension} we have the identity
\begin{equation}
\frac{1}{\pi}\int_{0}^{\infty}\,d\eta_1\int_{-\infty}^{\infty}\,d\tau_1 \frac{1}{x-\tau_1-i\eta_1}\overline{\partial}\Psi_{f_n}(\tau_1,\eta_1) = f_n(x)+iH[f_n](x) \label{genextension}.
\end{equation}
It is well known that $H$ is a bounded operator from $L^2(\mathbb{R})$ to itself, therefore if we take the limit as $n\to \infty$ on both sides of \eqref{genextension} (and note that $f_n \to f$ pointwise everywhere and in $L^2(\mathbb{R})$) we have
\begin{equation}
	\lim_{n\to\infty}\frac{1}{\pi}\int_{0}^{\infty}\,d\eta_1\int_{-\infty}^{\infty}\,d\tau_1 \frac{1}{x-\tau_1-i\eta_1}\overline{\partial}\Psi_{f_n}(\tau_1,\eta_1) = f(x)+iH[f](x),
\end{equation}
we check that the limit interchanges with integral by noting that
\begin{align}
	|\overline{\partial}\Psi_{f_n}(t,\eta)| \leq& 2|f_n'(t)- f_n'(t+\eta)||J(\eta)| + 2(|f(t)| + |f(t+\eta)|)|J'(\eta)|,\\
	|f_n'(t)- f_n'(t+\eta)| \leq& |\phi_n'(t+\eta)||f(t+\eta)-f(t)| + |\phi_n'(t+\eta)-\phi_n'(t)||f(t)|\\
				& +|\phi_n(t+\eta)||f'(t+\eta)-f'(t)| + |f'(t)||\phi_n(t+\eta)-\phi_n(t)|,
\end{align}
and that $\phi_n(t) = \phi(t/n)$ is inifinitely differentiable and therefore in $C^{k,\alpha}(\mathbb{R})$ for any $k\in \mathbb{N}$, $\alpha\in(0,1)$. So we may bound for large $t$, $|\overline{\partial}\Psi_{f_n}(t,\eta)|$ by $C \eta^{\alpha\sigma}|t|^{-(1+\beta)(1-\sigma)}$ for some constant $C$ and any $\sigma \in(0,1)$, it follows by the dominated convergence theorem that we may interchange limit with integral to obtain
\begin{equation}
	\frac{1}{\pi}\int_{0}^{\infty}\,d\eta_1\int_{-\infty}^{\infty}\,d\tau_1 \frac{1}{x-\tau_1-i\eta_1}\overline{\partial}\Psi_{f}(\tau_1,\eta_1) = f(x)+iH[f](x). \label{genextension2}
\end{equation}

The Fourier transform of the Hilbert transform is given by
\begin{equation}
\widehat{H[f]}(k) = -i\,\mathrm{sgn}(k)\hat{f}(k). \label{hatH}
\end{equation}
Hence, inserting \eqref{genextension2} into \eqref{covcomp} and applying \eqref{hatH} yields the limiting covariance structure
\begin{equation}
\begin{split}
\mathbb{E}(X(f_1)X(f_2)) &= \frac{1}{8\pi}\int_{-\infty}^{\infty}\,dk\,|k|\hat{f_1}(k)\overline{\hat{f_2}(k)}|1-i\,\mathrm{sgn}(k)|^{2}\,dk+\mathrm{c.c.}\\
&=\frac{1}{2\pi}\int_{-\infty}^{\infty}\,dk\,|k|\hat{f_1}(k)\overline{\hat{f_2}(k)}
\end{split}
\end{equation}
\end{proof}

\subsection{Functions supported on the real line}
\label{sec:realsupp}
The main goal of this subsection is to remove the assumption of compact support from the functions $f$ in Theorem \ref{thm:c2c} subject to the following decay condition on $f$ and $f'$: for some $\beta>0$ and $|x|$ large enough, $f(x)$ and $f'(x)$ are $O(|x|^{-1-\beta})$. This will complete the proof of the corresponding statement in our main Theorem \ref{th:maintheorem}. We begin by approximating $f$ by a compactly supported function whose support grows at a rate $O(d_{N})$ as $N \to \infty$. The support of the test function can now extend over the edges of the spectrum, so that the resolvent bounds of Proposition \ref{prop:bulkBound} cannot be applied. The goal of this subsection is to prove the following 
\begin{theorem}
\label{thm:c2csupp}
Let $d_{N} = N^{\gamma}$ with $0 < \gamma < 1/3$. If for some $\alpha>0$ and $\beta>0$, we have $f \in C^{1,\alpha}(\mathbb R)$ where $f(x)$ and $f'(x)$ decay faster than $|x|^{-1-\beta}$ for large $|x|$, then the random variable $X^{\mathrm{meso}}_{N}(f)$ converges in distribution to a Gaussian random variable with variance given by \eqref{covarc2c} of Theorem \ref{thm:c2c}. Moreover, the multidimensional version stated in Theorem \ref{thm:c2c} continues to hold for functions of this class.
\end{theorem}
We will prove this Theorem by means of the following Lemma and two Propositions.
\begin{lemma}
\label{lem:cutofflemma}
Let $\phi_N(x)$ denote a smooth cutoff function equal to $1$ in when $|E+x/d_{N}| \leq 2$ and equal to $0$ when $|E+x/d_{N}| \geq 4$. Let $f_{N}(x) := f(x)\phi_{N}(x)$. Then $X^{\mathrm{meso}}_{N}(f) = X^{\mathrm{meso}}_{N}(f_{N})+o_{L^{2}}(1)$.
\end{lemma}

\begin{proof}
	We follow the same technique here as in Section 4 of \cite{SW13}. Note that $\tilde{X}_N^{\mathrm{meso}}((1-\phi_N)f)$ is only non-zero when $\lambda_{1} < -4$ or $\lambda_{N} > 4$, so that
	\begin{equation}
		\mathbb{P}\left( \left|\sum_i\left( (1-\phi_N)f \right)(d_N\lambda_i)\right| > 0 \right) \leq\mathbb{P}\left( \lambda_1 > 4 \right) + \mathbb{P}\left( \lambda_N < -4 \right),
	\end{equation}
	these quantities are bounded by $e^{-N^c}$ by \cite[Lemma 7.2]{EYY12}. Thus
	\begin{align}
		\mathbb{E}|X_N^{\mathrm{meso}}\left( (1-\phi_N)f \right)|^2 &= 2 \int_0^{N\|f\|_{L^\infty}} x\mathbb{P}\left( \left|\sum_i (1-\phi_N)f(d_N\lambda_i)\right| \geq x \right)\,dx, \\
		&\leq N^2 \|f\|^2_{L^\infty(\mathbb{R})} \mathbb{P}\left( \left| \sum_i (1-\phi_N)f(d_N\lambda_i)\right| > 0 \right),\\
	&\leq N^2 e^{-N^c}\|f\|^2_{L^\infty(\mathbb{R})} \to 0;
	\end{align}
\end{proof}

Now we apply the Helffer-Sj\"ostrand formula \eqref{genint} to the function $f_{N}$, obtaining
\begin{equation}
\tilde{X}^{\mathrm{meso}}_{N}(f_{N}) = \frac{1}{\pi}\re\int_{0}^{1}\int_{\mathbb{R}}V_{N}(\tau+i\eta)\tilde{\chi}_{N}(\tau)\overline{\partial} \Psi_{f_{N}}(\tau,\eta)\,d\tau\,d\eta 
\end{equation}
where $\tilde{\chi}_{N}(\tau)$ is the indicator function of the region $|E+\tau/d_{N}| \leq 4$. Repeating the derivation of \ref{swb1}, we see that  
\begin{equation}
\int_{0}^{d_{N}/N}\int_{\mathbb{R}}\,\mathbb{E}(|V_{N}(\tau+i\eta)|\tilde{\chi}_{N}(\tau))|\overline{\partial} \Psi_{f_{N}}(\tau,\eta)|\,d\tau\,d\eta \to 0
\end{equation} 
on all scales of the form $d_{N} = N^{\gamma}$ with $0 < \gamma < 1$. To apply Theorem \ref{thm:weakconvInt}, we replace the $N$-dependent $f_{N}$ with $f$, noting that $\overline{\partial}\Psi_{f_{N}}-\overline{\partial}\Psi_{f} = \overline{\partial}\Psi_{f_{N}-f}$, where $f_{N}(\tau)-f(\tau) = f(\tau)(\phi_{N}(\tau)-1)$ is supported on the region $|\tau/d_{N}-E| \geq 4$. 
\begin{proposition}
\label{prop:addsupp}
We have
\begin{equation}
\frac{1}{\pi}\re\int_{0}^{1}\int_{\mathbb{R}}V_{N}(\tau+i\eta)\overline{\partial} \Psi_{f_{N}-f}(\tau,\eta)\,d\tau\,d\eta = o_{L^{1}}(1). \label{addsuppint}
\end{equation}
\end{proposition}
We postpone the proof of this Proposition until the end of this subsection, where it will follow from a more general argument. Thus $X^{\mathrm{meso}}_{N}(f) = I(f,V_{N})+o_{L^{1}}(1)$ where
\begin{equation}
I(f,V_N) =  \frac{1}{\pi}\re\int_{0}^{1}\int_{-\infty}^{\infty}V_{N}(\tau+i\eta) \chi_{N}(\eta)\tilde{\chi}_{N}(\tau)\overline{\partial} \Psi_{f}(\tau,\eta)\,d\tau d\eta.
\end{equation}
To show that $I(f,V_{N})$ converges to $I(f,\Gamma'^{+}_{0})$, thus completing the proof of Theorem \ref{thm:c2csupp}, it remains to check the following tightness result.
\begin{proposition}
\label{prop:tight2}
Consider the domain $D = [0,1] \times \mathbb{R}$. We have the following estimates:
\begin{align}
\inf_{B\subset \mathbb{H}, \lambda(B)<\infty} \limsup_{N\to\infty}\int_{D\setminus B}\mathbb{E}(|V_{N}(t+i\eta)|\chi_{N}(\eta)\tilde{\chi}_{N}(\tau))|\overline{\partial}\Psi_{f}(t,\eta)|\,d\eta\,dt=0 \label{limB},\\
\lim_{K \to \infty}K^{-\delta}\lim_{N \to \infty}\int_{D}\mathbb{E}(|V_{N}(t+i\eta)|^{1+\delta}\chi_{N}(\eta)\tilde{\chi}_{N}(\tau))|\overline{\partial}\Psi_{f}(t,\eta)|^{1+\delta}\,d\eta\,dt &= 0. \label{limK}
\end{align}
\end{proposition}

\begin{proof}
First we give some bounds on $|\overline{\partial}\Psi_{f}(t,\eta)|$. By our assumptions on $f$, we have
\begin{equation}
	|f'(t+\eta)-f'(t)|\leq \min(C_1\eta^\alpha, C_2 |t|^{-(1+\beta)}) \leq c \eta^{\alpha\sigma}|t|^{-(1+\beta)(1-\sigma)},
\end{equation}
for any $\sigma \in (0,1)$ and $c>0$ a constant independent of $t$ and $\eta$. Hence by construction we have the bound
\begin{equation}
 |\overline{\partial}\Psi_f(t,\eta)| \leq c\eta^{\alpha\sigma}|t|^{-(1+\beta)(1-\sigma)} \label{psibound}
\end{equation}
for large $|t|$ and small $\eta$. We proceed by splitting the integration in \eqref{limB} and \eqref{limK} into the regions $D_{\mathrm{bulk}} = [0,1] \times \{\tau : |E+\tau/d_{N}|\leq 2\}$ and $D_{\mathrm{out}} = [0,1] \times \{\tau : 2 \leq |E+\tau/d_{N}| \leq 4$. Starting with region $D_{\mathrm{bulk}}$, the variance bound of Proposition \ref{prop:bulkBound} is applicable and we see that $\mathbb{E}|V_{N}(t+i\eta)|^{1+\delta} \leq c_1\eta^{-(1+\delta)}$ uniformly in $t$. Combined with \eqref{psibound} we see that the integrands of \eqref{limB} and \eqref{limK} restricted to $D_{\mathrm{bulk}}$ are dominated by an integrable function and the proof proceeds as in the compactly supported case of Theorem \ref{thm:c2c}.

It remains to bound the contribution to the integrals on the domain $D_{\mathrm{out}}$. Here we can exploit the decay of the test function $f$ and show that the inner $\limsup$ over $N$ will already be zero in \eqref{limB} and \eqref{limK}. Therefore it suffices to take $B=\emptyset$ and $K=1$. Then the variance bound of Proposition \ref{prop:sw} applies, yielding
\begin{equation}
\begin{split}
	&\int_{D_{\mathrm{out}}}\mathbb{E}(|V_{N}(z)|^{1+\delta})|\overline{\partial} \Psi_{f}(t,\eta)|^{1+\delta}\,dt\,d\eta\\ 
	&\leq cd_{N}^{\epsilon(1+\delta)}\int_{D_{\mathrm{out}}}\, \eta^{(-1-\epsilon)(1+\delta)+\sigma \alpha}|t|^{(-1-\beta)(1-\sigma)(1+\delta)}\,dt\,d\eta\\
	&\leq cd_{N}^{\epsilon(1+\delta)-(1+\beta)(1-\sigma)(1+\delta)+1}
\label{intdout}
\end{split}
\end{equation}
where we choose $\epsilon$ and $\delta$ so small that the integral over $\eta$ is finite while the integral over $t$ goes to zero as $N \to \infty$. Indeed, if $\epsilon < \sigma\alpha/3$ and $\delta < \sigma\alpha/3$ then $(-1-\epsilon)(1+\delta)+\sigma\alpha > -1$ and the $\eta$ integral is finite. In the integral over $t$, we choose $\sigma < \delta/(1+\delta)$, $\epsilon < \beta/(1+\delta)$ and deduce that $\epsilon(1+\delta)-(1+\beta)(1-\sigma)(1+\delta)+1<0$. To make the bounds work simultaneously we take $\epsilon < \min\{\sigma\alpha/3,\beta/(1+\delta)\}$. We conclude that the limit of \eqref{intdout} is zero. To prove Proposition \ref{prop:addsupp} notice that the integrand is supported on $D_{\mathrm{out}}$ with the same regularity conditions on $f$. Hence an identical calculation to that given in \eqref{intdout} shows that \eqref{addsuppint} converges to zero in $L^1$ as $N \to \infty$. This completes the proof of Propositions \ref{prop:tight2} and \ref{prop:addsupp}. Consequently, by means of \ref{thm:weakconvInt}, this also completes the proof of Theorem \ref{thm:c2csupp}.
\end{proof}

\begin{corollary}
\label{cor:arzasc}
The sequence of stochastic processes $V_{N}(z)$ with $z \in \mathbb{H}$ is tight in the space of continuous functions on any $N$-independent rectangle in the upper half-plane $\mathbb{H}$.
\end{corollary}

\begin{proof}
It suffices to verify the Arzela-Ascoli criterion:
\begin{equation}
\mathbb{E}|V_{N}(z_1)-V_{N}(z_2)|^{2} \leq C|z_1-z_2|^{2},
\end{equation}
for $z_1$ and $z_2$ in some $N$-independent rectangle in $\mathbb{H}$ and $C$ a constant depending only on the vertices of the rectangle (i.e. not on $N$). To prove this, note that
\begin{equation}
\frac{1}{x-u_1-iv_1}-\frac{1}{x-u_2-iv_2} = \frac{(u_1-u_2)+i(v_1-v_2)}{(x-u_1-iv_1)(x-u_2-iv_2)}
\end{equation}
implies
\begin{equation}
\mathbb{E}|V_{N}(z_1)-V_{N}(z_2)|^{2} = |z_1-z_2|^{2}\mathbb{E}|\tilde{X}^{\mathrm{meso}}_{N}(h)|^{2}
\end{equation}
where $h(x) = ((x-u_1-iv_1)(x-u_2-iv_2))^{-1}$ is a smooth function with decay $h(x) = O(|x|^{-2})$ as $|x| \to \infty$. Hence the techniques of the present subsection are applicable with $\alpha=\beta=1$. Indeed, after replacing $h$ with a smooth cut-off $h_{N}$ as in Lemma \ref{lem:cutofflemma}, an application of formula \eqref{genint} followed by Cauchy-Schwarz leads to $\mathbb{E}|\tilde{X}^{\mathrm{meso}}_{N}(v)|^{2} \leq I_{v}^{2}$ where
\begin{equation}
I_{v} := \frac{1}{\pi}\int_{0}^{1}\int_{-\infty}^{\infty}\,\sqrt{\mathbb{E}|V_{N}(\tau+i\eta)|^{2}}\,|\overline{\partial}\Psi_{h_{N}}(\tau,\eta)|\,d\tau\,d\eta
\end{equation}
Following the proof of Proposition \ref{prop:tight2} we easily deduce that $I_{v}$ is uniformly bounded in $N$ for fixed $u_1,v_1,u_2,v_2 \in \mathbb{H}$.
\end{proof}

\appendix
\renewcommand{\thelemma}{\Alph{section}\arabic{lemma}}
\section{Weak Convergence results}
\begin{theorem}[Martingale Central Limit Theorem]{\cite[Theorem 35.12]{billingsley}}\label{thm:martclt}
	Let $M_{k,N}$, $1\leq k \leq N$, $N \geq 1$ be a sequence of zero-mean, square-integrable martingales adapted to the filtration $\mathcal{F}_{k,N}$, and let $\mathcal{F}_{0,N}$ denote the trivial $\sigma$-field. Let $Y_{k,N}$ denote the martingale difference sequence $Y_{k,N} = M_{k,N} - M_{k-1,N}$ where $1 \leq k \leq N$. Suppose the following conditions hold
	\begin{align}
		\text{for all $\epsilon > 0$} \quad \sum_{k=1}^N \mathbb{E}[Y_{k,N}^2 \mathbf{1}_{|Y_{k,N}|>\epsilon}|\mathcal{F}_{k-1,N}] \to 0 \quad \text{in probability}, \\
		\sum_{k = 1}^N \mathbb{E}[ Y_{k,N}^2 | \mathcal{F}_{k-1,N}] \to \sigma^2 \quad \text{in probability},
	\end{align}
	where $\sigma^2$. Then $M_{N,N}$ converges in distribution to a Gaussian random variable with mean 0 and variance $\sigma^2$. 
\end{theorem}

\begin{theorem}{\cite[Theorem 1; Lemma 1]{cremers1986weak}}
	\label{thm:weakconvInt}
	Let $\Phi$ be a space of measurable functions $\phi(z,w):\mathbb{H}\times \mathbb{C} \to \mathbb{C}$ such that $\phi(z,\cdot)$ is continuous for all $z\in \mathbb{H}$. Let $M$ be a set of measurable functions $x : \mathbb{H} \to \mathbb{C}$ such that
	\begin{equation}
		\int |\phi(z,x(z))|\,dz < \infty,\quad \phi \in \Phi,
	\end{equation}
	and define the functional
	\begin{equation}
		\ell_\phi(\xi_N) = \int_{\mathbb{H}}\phi(z,\xi_N(z))\,dz.
	\end{equation}
	Suppose $\{\xi_N: N\in \mathbb{N}_0\}$ is a sequence of stochastic processes  $\xi_N(z) : \mathbb{H} \to \mathbb{C}$, with paths in $M$. If the finite dimensional distributions of $\xi_N$ converge weakly to those of $\xi_0$ Lebesgue almost everywhere in $\mathbb{H}$ and for all $\epsilon > 0$ and $\phi \in \Phi$:
	\begin{align}
		\lim_{K \to \infty} \limsup_{N \to \infty}\mathbb{P}\left(\int_{\mathbb{H}}(|\phi(z,\xi_N(z))|-K)^+\,dz \geq \epsilon \right) &= 0,\\
	\inf_{B\subset \mathbb{H}, \lambda(B)<\infty} \limsup_{N\to\infty}\mathbb{P}\left(\int_{\mathbb{H} - B}|\phi(z,\xi_N(z))|\,dz \geq \epsilon  \right) &= 0,
	\end{align}
where $\lambda$ is the Lebesgue measure on $\mathbb{H}$, then for $\phi_1,\dots, \phi_k \in \Phi$, we have that 
	\begin{equation}
		(\ell_{\phi_1}(\xi_N),\dots,\ell_{\phi_k}(\xi_N)) \Rightarrow (\ell_{\phi_1}(\xi_0),\dots, \ell_{\phi_k}(\xi_0)).
	\end{equation}
\end{theorem}
\section{Concentration inequalities and bounds on the resolvent}
In this section we record some important bounds required in Section \ref{se:res}. We remind the reader the notation used in that Section 
\begin{equation}
	\delta^{k,n}_{N}(z) := h_{k}^{\dagger}G_{k}(z)^{n}h_{k}-N^{-1}\tr(G_{k}(z_{1})^{n}) \label{deltaknapp}
\end{equation}
where $h_{k}$ denotes the $k^{\mathrm{th}}$ column of $\mathcal{H}$ with $k^{\mathrm{th}}$ element removed. The matrix $\mathcal{H}_{k}$ is defined as the matrix $\mathcal{H}$ with the $k^{\mathrm{th}}$ row and column erased and $G_{k}(z) := (\mathcal{H}_{k}-z)^{-1}$ is the resolvent of $\mathcal{H}_{k}$.
\begin{lemma}
\label{d1bound}
Let $z = E+\frac{\tau+i\eta}{d_{N}}$. Then for fixed $\eta>0$ and $\tau \in \mathbb{R}$ and all $q\geq 1$, there are positive constants $c_1,c_2$ such that
	\begin{align}
		&\mathbb{E}|\delta^{k,1}_{N}(z)|^{q} \leq c_1\left(\frac{N\eta}{d_N}\right)^{-\frac{q}{2}} \label{hswright}\\
		&\mathbb{E}\left|\frac{1}{z + \frac{1}{N}\tr(G_k(z))}\right|^{q} \leq c_2 \label{boundedexps}
	\end{align}
	for all $k$, $N$. 
\end{lemma}

\begin{proof}
Inequality \eqref{hswright} can be found in Proposition 3.2 of \cite{CMS14} as a consequence of the Hanson-Wright large deviation inequality \cite{RV13}. Such concentration inequalities appear in several other works concerning the local semi-circle law for Wigner matrices, see \textit{e.g.} \cite{ESY08,ESY09,EYY12,EYY12rig}. To prove \eqref{boundedexps}, set $\tilde{b}_{k} := (z+N^{-1}\mathrm{Tr}(G_{k}))^{-1} = (z+s(z) + \Lambda_{k})^{-1}$ where $\Lambda_{k} = N^{-1}\mathrm{Tr}(G_{k})-s(z)$. Now we use the elementary identity
\begin{equation}
\label{taylorbk}
\tilde{b}_{k}:=\frac{1}{z+s(z) +\Lambda_{k}} = \sum_{i=0}^{p}\frac{\Lambda_{k}^{i}}{(z+s(z))^{i+1}} + \frac{\Lambda_{k}^{p+1}}{(z+s(z))^{p}(z+N^{-1}\mathrm{Tr}(G_{k}))}
\end{equation}
Then it is known that $|z+s(z)|^{-1} \leq 1$ uniformly on the upper-half plane $z \in \mathbb{C}^{+}$ (see \textit{Eq. 8.1.19} in \cite{BS10}). Furthermore, we have the trivial bound $|z+N^{-1}\mathrm{Tr}(G_{k})|^{-1} \leq (\eta/d_N+N^{-1}\mathrm{Im}\mathrm{Tr}(G_{k}))^{-1} \leq d_N/\eta$. This gives us
\begin{equation}
	|\tilde{b}_{k}| \leq \sum_{i=0}^{p}|\Lambda_{k}|^{i}+\frac{d_N}{\eta}|\Lambda_{k}|^{p+1}
\end{equation}
On the other hand, by rigidity arguments we know that $\mathbb{E}|\Lambda_{k}|^{q} \leq C'(N\eta/d_N)^{-q}$, so for example $\mathbb{E}|\tilde{b}_{k}| \leq 1+O((N\eta/d_N)^{-1})+O((d_N/\eta)(N\eta/d_N)^{-p-1})$. This last error term can be made $o(1)$ by choosing $p$ large enough (setting $d_N=N^{\alpha}$, one finds the condition $p+1 > \alpha/(1-\alpha)$)
\end{proof}

\begin{lemma}
\label{lem:derivbound}
	Let $z = E + \frac{\tau + i \eta}{d_N}$. For all $q \geq 1$ with fixed $\eta>0$ and $\tau \in \mathbb{R}$, there are positive constants $c,C$ such that
	\begin{align}
		&\mathbb{E}|d_N^{-1}\delta_N^{k,2}(z)|^{q} < C\left(\frac{N}{d_N}\right)^{-q/2}\eta^{-\frac{3q}{2}},\\
		&\mathbb{E}\left|\frac{\eta}{d_N} \tr(G_{k}^{2}(z))\right|^{q} \leq c\max\left(\frac{\eta}{d_N},\left(\frac{N\eta}{d_N}\right)^{-1}\right)^{q},
	\end{align}
	for all $k$, $N$.
\end{lemma}
\begin{proof}
Since $G_{k}(z)$ is an analytic function of $z$ in the upper half plane and $\frac{d}{dz}G_{k}(z) = G^{2}_{k}(z)$, we write
\begin{equation}
	d_N^{-1}\delta_N^{k,2} = d_N^{-1}\frac{d}{dz}\left(h_k^\dagger G_{k} h_{k}-\frac{1}{N}\tr(G_{k})\right) = \frac{1}{2\pi i}\oint_{S_{z}}\frac{1}{(z-w)^{2}}\left(h_{k}^{\dagger}G_{k}(w)h_k-\frac{1}{N}\tr(G_{k}(w))\right)dw
\end{equation}
where $S_{z}$ is a small circle of radius $\eta/2d_N$ around the point $z$. Note that $S_{z}$ is of distance $O(\eta/d_N)$ from the real axis and that $|z-w| = \eta/2d_N$. By Holder's inequality and Lemma \ref{d1bound}, we have
\begin{align}
	\mathbb{E}|d_N^{-1}\delta_N^{k,2}|^{q} &\leq \frac{1}{(2\pi d_N)^{q}}\oint_{S_{z}}\ldots\oint_{S_{z}}\mathbb{E}\prod_{i=1}^{q}|\delta_N^{k,1}(w_i)|\frac{1}{|z-w_{i}|^{2}}|dw_{i}|\\
	&\leq \frac{1}{(2\pi d_N)^{q}}\oint_{S_{z}}\ldots\oint_{S_{z}}\prod_{i=1}^{q}\mathbb{E}(|\delta_N^{k,1}(w_i)|^{q})^{1/q}\frac{1}{|z-w_{i}|^{2}}|dw_{i}|\\
	&\leq C\left(\frac{N\eta}{d_N}\right)^{-q/2}\frac{1}{(2\pi d_N)^{q}}\left(\int_{S_{z}}\frac{1}{|z-w|^{2}}|dw|\right)^{q}\\
	&=C\left(\frac{N}{d_N}\right)^{-q/2}\eta^{-\frac{3q}{2}}.
\end{align}
As for $\tr(G_{k}^{2})$ we similarly have
\begin{equation}
\label{cauchint}
\frac{\eta}{d_N}\tr(G_{k}^{2}) = \frac{\eta}{d_N}\frac{d}{dz}\tr G_{k} = \frac{\eta}{d_N}s'(z) + \frac{\eta}{2\pi i d_N}\int_{S_{z}}\frac{1}{(w-z)^{2}}(N^{-1}\tr G_{k}(w)-m(w))dw
\end{equation}
where $S_{z}$ is a small circle of radius $\eta/2d_N$ with center $z = z(s)$. Then using the known rigidity estimates we get $\mathbb{E}|N^{-1}\tr(G_{k}(w))-s(z)|^{q} \leq (N\eta/d_N)^{-q}$ which we use to estimate \eqref{cauchint}. Then
\begin{align*}
	\mathbb{E}\left|\frac{\eta}{d_N} \tr(G_{k}^{2})\right|^{q} &\leq \sum_{i=0}^{q}\binom{q}{i}\left|\frac{\eta}{d_N} s'(z)\right|^{q-i}\mathbb{E}\left|\frac{\eta}{2\pi d_N}\int_{S_{z}}\frac{1}{(w-z)^{2}}(N^{-1}\tr G_{k}(w)-m(w))dw\right|^{q},\\
	&\leq \sum_{i=0}^{q}\binom{q}{i}\left|\frac{\eta}{d_N} s'(z)\right|^{q-i}\left(c_{i} \frac{N\eta}{d_N}\right)^{-i},\\
&\leq c\max\left(\frac{\eta}{d_N},\left(\frac{N\eta}{d_N}\right)^{-1}\right)^{q},
\end{align*}
as required.
\end{proof}

\begin{lemma}
\label{le:schurident}
The identity \eqref{errorsgkmt} holds.
\end{lemma}
\begin{proof}
We omit the explicit $z$-dependence on the resolvent and make use of the quantities given in $\delta^{k,n}_{N}$ in \eqref{deltaknapp} and two further quantities
\begin{equation}
\tilde{b}_{k} := \frac{1}{z+N^{-1}\tr G_{k}}, \qquad G_{kk} = \frac{1}{\mathcal{H}_{kk}-z-h_{k}^{\dagger}G_{k}h_{k}}
\end{equation}
Using the identity $z+N^{-1}\tr G_{k} = (-\delta^{k,1}_{N}+\mathcal{H}_{kk}-G_{kk}^{-1})$ we obtain
\begin{align}
&\frac{1+h_{k}^{\dagger}G_{k}^{2}h_{k}}{\mathcal{H}_{kk}-z-h_{k}^{\dagger}G_{k}^{2}h_{k}}-\frac{1+N^{-1}\tr G_{k}^{2}}{-z-N^{-1}\tr G_{k}} \label{gkdecomp}\\
&=\frac{(1+h_{k}^{\dagger}G_{k}^{2}h_{k})(\mathcal{H}_{kk}-\delta^{k,1}_{N})}{\tilde{b}_{k}^{-1}G_{kk}^{-1}}-\frac{G_{kk}^{-1}(1+h_{k}^{\dagger}G_{k}^{2}h_{k})}{G_{kk}^{-1}\tilde{b}_{k}^{-1}}+\frac{G_{kk}^{-1}(1+N^{-1}\tr G_{k}^{2})}{G_{kk}^{-1}\tilde{b}_{k}^{-1}}\\
&=\frac{G_{kk}(1+h_{k}^{\dagger}G_{k}^{2}h_{k})(\mathcal{H}_{kk}-\delta^{k,1}_{N})}{\tilde{b}_{k}^{-1}}-\delta^{k,2}_{N}\tilde{b}_{k}\\
&= -\tilde{b}_{k}^{2}(1+h_{k}^{\dagger}G_{k}^{2}h_{k})(\mathcal{H}_{kk}-\delta^{k,1}_{N})+\tilde{b}_{k}^{2}(1+h_{k}^{\dagger}G_{k}^{2}h_{k})(\mathcal{H}_{kk}-\delta^{k,1}_{N})^{2}-\delta^{k,2}_{N}\tilde{b}_{k} \label{gkkexpand}\\
&=-\tilde{b}_{k}^{2}(1+N^{-1}\tr G_{k}^{2})(\mathcal{H}_{kk}-\delta^{k,1}_{N})-\delta^{k,2}_{N}\tilde{b}_{k} \label{gk2expand}\\
&-\tilde{b}_{k}^{2}\delta^{k,2}_{N}(\mathcal{H}_{kk}-\delta^{k,1}_{N})+\tilde{b}_{k}^{2}(1+h_{k}^{\dagger}G_{k}^{2}h_{k})(\mathcal{H}_{kk}-\delta^{k,1}_{N})^{2} \label{errorsgk}
\end{align}
where to obtain \eqref{gkkexpand} we expanded using the simple identity $G_{kk} = -\tilde{b}_{k}-\tilde{b}_{k}G_{kk}(\delta^{k,1}_{N}-\mathcal{H}_{kk})$. Then to obtain \eqref{gk2expand} from \eqref{gkkexpand} we used that $h_{k}^{\dagger}G_{k}^{2}h_{k} = N^{-1}\tr G_{k}^{2}+\delta^{k,2}_{N}$. The two terms in \eqref{gk2expand} combine as an exact derivative
\begin{equation}
\frac{\partial}{\partial z}\tilde{b}_{k}(\mathcal{H}_{kk}-\delta^{k,1}_{N}) = -\tilde{b}_{k}^{2}(1+\tr G_{k}^{2})(\mathcal{H}_{kk}-\delta^{k,1}_{N})-\delta^{k,2}_{N}\tilde{b}_{k}
\end{equation}
while the remaining two terms in \eqref{errorsgk} combine to give the error term $\epsilon_{k,N}(z)$ in \eqref{errorsgkmt}. Finally, we conclude the proof of the identity by applying $(\mathbb{E}_{k}-\mathbb{E}_{k-1})$ to \eqref{gkdecomp}, noting that the second term vanishes.
\end{proof}

\begin{proposition}
	\label{prop:bulkBound}
	Fix $\tilde{\eta} > 0$. Then under the same conditions as Theorem \ref{thm:optimalbound} we have that there exists a positive constants $N_0$, $M_0$, $C$, $c_0$, and $c_1 = c_1(C,c)$ such that
\begin{equation}
\label{innervarbound}
\mathbb{E}|V_N(z)|^2 = \mathrm{Var}\bigg\{d_N^{-1}\tr G(E+z/d_N)\bigg\} \leq c_1\eta^{-2},
\end{equation}
for all $N > N_0$ such that $N\eta/d_N \geq M_0$, $\eta/d_N \leq \tilde{\eta}$, and $|E+ t/d_N| \leq 2 + \eta/d_N$.
\end{proposition}
\begin{proof}
By the triangle inequality, we can write
\begin{equation}
\begin{split}
	&\mathbb{E}|V_N(z)|^2 = \mathbb{E}|(N/d_{N})(s_{N}(E+z/d_N)-\mathbb{E}s_{N}(E+z/d_N))|^{2}\\
&\leq 2\mathbb{E}|(N/d_{N})(s_{N}(E+z/d_N)-s(E+z/d_N))|^{2}+ 2|(N/d_{N})\mathbb{E}(s_{N}(E+z/d_N)-s(E+z/d_N))|^{2}\\
&\leq4\mathbb{E}|(N/d_{N})(s_{N}(E+z/d_N)-s(E+z/d_N))|^{2}
\end{split}
\end{equation}
where the last line follows from Jensen's inequality. The latter expectation is by definition the integral
\begin{multline}
\label{intvarsn}
\int_{0}^{\infty}2u\mathbb{P}(|s_{N}(E+z/d_N)-s(E+z/d_N)|>ud_{N}/N)\,du\\
= \eta^{-2}\int_0^\infty 2K \mathbb{P}\left(|s_N(E+z/d_N)-s(E+z/d_N)| > \frac{K d_N}{\eta N}\right) \,dK
\end{multline}
It suffices to bound the contribution to the above integral when $K \geq 1$. We apply Theorem $\ref{thm:optimalbound}$ with \textit{e.g.} $q=3$ to obtain a convergent estimate.
\end{proof}

\section{Function extension}
The following Lemma is useful for proving results about the linear statistics \eqref{lsmeso} of a random matrix via analogous results for the resolvent. It can be found in \textit{Proof of Lemma 5.5.5} of the book \cite{AGZ09}.
\begin{lemma}
	\label{lem:extension}
	Let $f:\mathbb{R} \to \mathbb{R}$ be a compactly supported function with first derivative continuous (denote it by $f \in C_c^1(\mathbb{R})$) and consider an extension $\Psi_f:\mathbb{R}^2 \to \mathbb{C}$ that inherits the same regularity and such that $\Psi_f(x,0) = f(x)$ for all $x$ and $\im(\Psi_f(x,0)) = 0$. Further assume that $\overline{\partial}\Psi_f(x,y) = O(y)$ as $y \to 0$. Then one has
	\begin{equation}
		f(\lambda) + iH[f] = \frac{1}{\pi}\int_0^\infty \int_{-\infty}^\infty \frac{\overline{\partial} \Psi_f(x,y)}{\lambda - x - iy}\,dx\,dy,
	\end{equation}
	where $\overline{\partial}:= \frac{\partial}{\partial x} + i \frac{\partial}{\partial y}$, and $H:L^2(\mathbb{R}) \to L^2(\mathbb{R})$ is the \textit{Hilbert Transform} of $f$:
	\begin{equation}
		H[f] = \frac{1}{\pi}\mathrm{p.v.} \int_{-\infty}^{\infty}\,dt \frac{f(t)}{x-t}.
	\end{equation}
\end{lemma}
\begin{remark}
	\label{rem:specificPsi}
	Our particular choice of $\Psi_f$ in this paper will be the following (see \cite{davies1995functional}):
	\begin{equation}
		\Psi_f (x,y) = (f(x) + i(f(x+y)-f(x)))J(y),
	\end{equation}
	where $J(y)$ is a smooth function of compact support, equal to 1 in a neighborhood of 0 and equal to 0 if $y > 1$.
\end{remark}
\begin{proof}
We make use of the the substitution $x \to x+\lambda$ and compute the real part (denoting $\Psi_{f} = u+iv$)
\begin{align}
&-\frac{1}{\pi}\re\int_{0}^{\infty}dy\,\int_{-\infty}^{\infty}dx\,\frac{\overline{\partial}\Psi_{f}(x+\lambda,y)}{x+iy},\\
&=-\frac{1}{\pi}\int_{0}^{\infty}dy\, \int_{-\infty}^{\infty}dx\, \frac{x}{x^{2}+y^{2}}\left(\frac{\partial u(x,y)}{\partial x}-\frac{\partial v(x,y)}{\partial y}\right)+\frac{y}{x^{2}+y^{2}}\left(\frac{\partial u(x,y)}{\partial y}+\frac{\partial v(x,y)}{\partial x}\right),\\
&=-\frac{1}{\pi}\lim_{\epsilon \to 0}\int_{A_{\epsilon,R}}dx\,dy\, \frac{x}{x^{2}+y^{2}}\left(\frac{\partial u(x,y)}{\partial x}-\frac{\partial v(x,y)}{\partial y}\right)+\frac{y}{x^{2}+y^{2}}\left(\frac{\partial u(x,y)}{\partial y}+\frac{\partial v(x,y)}{\partial x}\right),
\end{align}
where $A_{\epsilon,R} = \{\epsilon < \sqrt{x^{2}+y^{2}} < R\}$ and $R$ is sufficiently large (by the compact support hypothesis). Changing to polar coordinates in the latter integral, a simple computation shows that the integral transforms as
\begin{align}
&-\frac{1}{\pi}\lim_{\epsilon \to 0}\int_{0}^{\pi}d\theta\,\int_{\epsilon}^{R}dr\,\frac{\partial u(r\cos(\theta),r\sin(\theta))}{\partial r}-\frac{\partial}{\partial \theta}\frac{v(r\cos(\theta),r\sin(\theta))}{r}\\
&=-\frac{1}{\pi}\lim_{\epsilon \to 0}\oint_{R_{\epsilon,R}}u(r\cos(\theta),r\sin(\theta))\,d\theta+\frac{v(r\cos(\theta),r\sin(\theta))}{r}\,dr\\
&=-\frac{1}{\pi}\lim_{\epsilon \to 0}\left(\int_{0}^{\pi}\,d\theta\,u(R\cos(\theta),R\sin(\theta))-\int_{0}^{\pi}\,d\theta\,u(\epsilon \cos(\theta),\epsilon \sin(\theta))+\int_{\epsilon < |r| < R} \frac{v(r,0)}{r}\,dr\right)
\end{align}
where in the second line we applied Green's theorem to reduce the double integral to an integral over the boundary of the positively oriented rectangle $R_{\epsilon,R}$ with vertices $(\epsilon,0),(R,0),(R,\pi),(\epsilon,\pi)$. Now choosing $R$ large enough and using the assumption $\mathrm{Im}(\Psi_{f}(r+\lambda,0))=v(r,0)=0$ we see that the above limit is equal to $u(0,0) = \Psi_{f}(\lambda,0) = f(\lambda)$. The proof for the imaginary part is similar.
\end{proof}

\bibliographystyle{plain}
\bibliography{meso}
\end{document}